\documentclass[12pt]{article}

\usepackage{amsmath,amssymb,bm,amsthm,mathrsfs,amscd}
\usepackage{here}
\usepackage{tikz}
\usepackage{color}
\theoremstyle{definition}
\newtheorem{Def}{Def}[section]
\newtheorem{Defi}[Def]{Definition}
\newtheorem{Them}[Def]{Theorem}
\newtheorem{Lem}[Def]{Lemma}
\newtheorem{Cor}[Def]{Corollary}
\newtheorem{Prop}[Def]{Proposition}
\newtheorem{Examp}[Def]{Example}

\DeclareMathOperator{\lcm}{lcm}
\DeclareMathOperator{\Mat}{Mat}
\DeclareMathOperator{\rank}{rank}
\DeclareMathOperator{\Atom}{Atom}

\numberwithin{equation}{section}

\title{Characteristic quasi-polynomials for deformations of Coxeter arrangements of types A, B, C, and D}
\author{Yusuke Mori\thanks{Department of Mathematics, Nagoya Institute of Technology, Aichi, 466-8555, Japan. Email: y.mori.404@nitech.jp}\quad
and\quad Norihiro Nakashima\thanks{Department of Mathematics, Nagoya Institute of Technology, Aichi, 466-8555, Japan. Email: nakashima@nitech.ac.jp (Corresponding author)}}

\date{}
\begin{document}
\maketitle

\begin{abstract}
Kamiya, Takemura, and Terao introduced a characteristic quasi-polynomial which enumerates the numbers of elements in the complement of hyperplane arrangements modulo positive integers.
In this paper, we compute the characteristic quasi-polynomials for specific arrangements which contain the Coxeter arrangements of types A, B, C, and D described by the orthonormal basis.
We also compute the characteristic quasi-polynomials for their deletion arrangements and we can show that they are factorized.
From this result, the poset generated by hypertori of the corresponding toric arrangement is an inductive poset.

\vspace{2mm}

\end{abstract}

\section{Introduction}
Kamiya, Takemura, and Terao \cite{KTT08, KTT11} proved, in two different ways, that the cardinality of the complement of a hyperplane arrangement modulo positive integers is a quasi-polynomial, called the characteristic quasi-polynomial.
The first proof is based on the theory of elementary divisors, whereas the second proof uses the theory of Ehrhart quasi-polynomials, studied by Beck and Zaslavsky \cite{Beck-Zaslavsky} in the framework of inside-out polytopes.
In these works, Kamiya, Takemura, and Terao also introduced a period of the quasi-polynomial, called the lcm period. Subsequently, various studies have investigated characteristic quasi-polynomials and their periods.

In general, the lcm period is not necessarily the minimum period; when the minimum period is strictly smaller, we say that a period collapse occurs. Higashitani, Tran, and Yoshinaga \cite{HTY2112} explicitly computed characteristic quasi-polynomials of certain non-central arrangements and exhibited concrete examples where period collapse occurs. In the same paper, they proved that in the central case, the lcm period always equals the minimum period.
Furthermore, the $k$-constituents of characteristic quasi-polynomials have been studied as characteristic polynomials of posets consisting of layers of abelian arrangements. Pagaria, Pismataro, Tran, and Vecchi \cite{PPTV2024} introduced a class of posets, called inductive posets, which allows computation of $k$-constituents while preserving their factorization property.\footnote{A poset $\mathcal{P}$ is also said to be inductive if every chain in $\mathcal{P}$ has an upper bound in $\mathcal{P}$. This terminology differs from the inductive posets defined in \cite{PPTV2024}.}

On the other hand, Coxeter arrangements and their deformations—such as extended Shi and extended Linial arrangements—have served as rich sources of examples in the study of characteristic quasi-polynomials. Kamiya, Takemura, and Terao \cite{KTT10} obtained a formula for the generating function of the characteristic quasi-polynomial of Coxeter arrangements, expressed via the basis of simple roots; this formula essentially derives from earlier results of Athanasiadis \cite{Athanasiadis1996} and Blass and Sagan \cite{Blass-Sagan}. In the first version \cite{KTT10-arxiv-v1} of \cite{KTT10}, they explicitly determined characteristic quasi-polynomials for the Coxeter arrangements of types B, C, and D described by the orthonormal basis. Yoshinaga \cite{Yoshinaga18b} determined the characteristic quasi-polynomial of the extended Shi arrangement and proved that it is a polynomial. The conjecture of Postnikov and Stanley \cite{Postnikov-Stanley} regarding the roots of characteristic polynomials of extended Linial arrangements was partially verified by Yoshinaga \cite{Yoshinaga18a} using expressions in terms of Eulerian polynomials, and was later completely resolved by Tamura \cite{Tamura}.

In this paper, we introduce two specific classes of hyperplane arrangements, defined in this work.
The first class consists of arrangements that include the Coxeter arrangement of type A described by the orthonormal basis as a special case, and the second consists of arrangements that include the Coxeter arrangements of types B, C, and D described by the orthonormal basis as special cases.
For these classes, we explicitly compute the corresponding characteristic quasi-polynomials by counting the number of points in their complements modulo positive integers.

Our focus on these classes is motivated by two considerations.
First, although characteristic quasi-polynomials have attracted attention in recent years, there are still not so many examples whose explicit forms are known, especially beyond classical Coxeter arrangements. Our classes provide infinite families of arrangements for which characteristic quasi-polynomials can be computed concretely.
Second, these classes are compatible with the deletion–restriction formula in the sense that it is easy to identify subclasses closed under restriction modulo integers. This enables us to construct sequences of subarrangements that preserve the condition that all $k$-constituents factor into linear polynomials. As a consequence, we prove that the posets generated by hypertori of the corresponding toric arrangements form inductive posets, providing new concrete examples in a field where research has only recently begun and where few explicit examples are known.

Finally, we also present examples of sequences of hyperplane arrangements in which the characteristic quasi-polynomials do not factorize into linear polynomials, illustrating boundaries of factorization phenomena within these structured classes.

The organization of this paper is as follows.
In Section \ref{sec-cox-quasi-charpoly}, we recall the definition of the characteristic quasi-polynomial and survey known results, including closed formulas for the characteristic quasi-polynomials of the Coxeter arrangements of types B, C, and D described by the orthonormal basis.
In Section \ref{sec-char-poly-deform-Cox}, we introduce the two classes of arrangements defined in this work and explicitly compute their characteristic quasi-polynomials.
Using these results together with the deletion–restriction formula, we construct sequences of subarrangements whose $k$-constituents are all factorized in Section \ref{sec-deleting-hyperplane}.
Finally, in Section \ref{sec-other-examples}, we compute the characteristic quasi-polynomials of certain subarrangements of Coxeter arrangements of types B and D and demonstrate that they do not factor into linear polynomials with integer coefficients.


\section{Preliminaries}\label{sec-cox-quasi-charpoly}
In this section, we fix notations and introduce some results for characteristic quasi-polynomials.
A function $\chi:\mathbb{Z}_{>0}\rightarrow\mathbb{Z}$ is called a quasi-polynomial if there exist a positive integer $\rho\in\mathbb{Z}_{>0}$ and polynomials $\chi^1(t),\chi^2(t),\dots,\chi^{\rho}(t)\in\mathbb{Z}[t]$ such that for $q\in\mathbb{Z}_{>0}$,
\begin{align*}
\chi(q)=
\begin{cases}
\chi^1(q)\quad&q\equiv 1\  mod\ \rho,\\
\chi^2(q)\quad&q\equiv 2\  mod\ \rho,\\
\quad\vdots&\quad\vdots\\
\chi^{\rho}(q)\quad&q\equiv \rho\ mod\ \rho.\\
\end{cases}
\end{align*}
In addition, we call $\chi$ a monic quasi-polynomial if $\chi^1(t),\chi^2(t),\dots,\chi^{\rho}(t)$ are monic polynomials.
The number $\rho$ is called a period and the polynomials $\chi^1(t),\chi^2(t),\dots,\chi^{\rho}(t)$ are called the $k$-constituents of the quasi-polynomial $\chi$.
A quasi-polynomial $\chi$ with a period $\rho$ is said to have the gcd property with respect to $\rho$ if the $k$-constituents $\chi^1(t),\chi^2(t),\dots,\chi^{\rho}(t)$ depend on $k$ only through $\gcd(\rho,k)$, i.e., $\chi^a(t)=\chi^b(t)$ if $\gcd(\rho,a)=\gcd(\rho,b)$.

Let $S\in\Mat_{m\times n}(\mathbb{Z})$ be an $m\times n$ matrix with integral entries and $\bm{b}\in\mathbb{Z}^n$ an integral vector.
Let $q\in\mathbb{Z}_{>0}$ and define $\mathbb{Z}_q:=\mathbb{Z}/q\mathbb{Z}$.
For $a\in\mathbb{Z}$, let $[a]_q:=a+q\mathbb{Z}\in\mathbb{Z}_q$ be the $q$ reduction of $a$.
For a matrix or vector $A$ with integral entries, let $[A]_q$ stand for the entry-wise $q$ reduction of $A$.
The matrix $S$ and the vector $\bm{b}$ define an arrangement $\mathscr{A}_q=\mathscr{A}_q(S,\bm{b})$ in $\mathbb{Z}_q^m$.
Indeed if $S=(\bm{h}_1\ldots \bm{h}_n)$ and $\bm{b}=(b_1,\dots,b_n)$ where $\bm{h}_1,\dots,\bm{h}_n$ are the columns of $S$, then the arrangement $\mathscr{A}_q$ is the collection of $H_{i,q}=\{\bm{x}=(x_1,\dots,x_m)\in\mathbb{Z}_q^{m}\mid \bm{x}[\bm{h}_i]_q=[b_i]_q\}$ for $i=1,\dots,n$.
We call the case $\bm{b}=\bm{0}\in\mathbb{Z}^m$ the central case, and call the case $\bm{b}\neq \bm{0}\in\mathbb{Z}^m$ the non-central case.
For $q\in\mathbb{Z}_{>0}$, let $\mathbb{Z}_q^{\times}:=\mathbb{Z}_q\setminus\{0\}$.
The complement $M_{\mathscr{A}_q}(q)$ of the arrangement $\mathscr{A}_q$ is defined by
\begin{align*}
M_{\mathscr{A}_q}(q):=\mathbb{Z}_q^m\setminus\bigcup_{H_q\in\mathscr{A}_q}H_q
=\left\{\bm{x}=(x_1,\dots,x_m)\in\mathbb{Z}_q^m\,\middle|\,\bm{x}[S]_q-[\bm{b}]_q\in\left(\mathbb{Z}_q^{\times}\right)^n\right\}.
\end{align*}
For convenience, when we consider arrangement over $\mathbb{Z}_q$, we correspond the pair $(S,\bm{b})$ to
\begin{align*}
\mathscr{A}=\{\{a_{1,i}x_1+\cdots+a_{m,i}x_m=b_i\}\mid i=1,\dots,n\},
\end{align*}
where $\bm{h}_i=\,^t(a_{1,i},\dots,a_{m,i})$ for any $i=1,\dots,n$.
Note that $\mathscr{A}$ differs from the hyperplane arrangement on $\mathbb{R}^m$.
For example, if $S= \begin{pmatrix} 2&3&0\\ 0&0&1 \end{pmatrix}$, then the hyperplanes $\{(x_1,x_2)\in\mathbb{R}^2\mid 2x_1=0\}$ and $\{(x_1,x_2)\in\mathbb{R}^2\mid 3x_1=0\}$ on $\mathbb{R}^2$ are coincide, but we write $\mathscr{A}=\{\{2x_1=0\},\{3x_1=0\},\{x_2=0\}\}$ if $\mathscr{A}$ is an arrangement over $\mathbb{Z}_q$.

We next define the lcm period and the characteristic quasi-polynomial.
For $\emptyset\neq J\subseteq \{1,\dots,n\}$, the matrix $S_J\in\Mat_{m\times|J|}(\mathbb{Z})$ is defined by the submatrix of $S$ consisting of the columns indexed by $J$.
Set $\ell(J):=\rank(S_J)$.
Let $e_{J,1}|e_{J,2}|\cdots|e_{J,\ell(J)}$ be the elementary divisors of $S_J$.
Then the lcm period of $S$ is defined by
\begin{align*}
\rho_S:=\lcm\left(e_{J,\ell(J)}\,\middle|\,\emptyset\neq J\subseteq \{1,\dots,n\}\right).
\end{align*}
We also write $\rho_{\mathscr{A}}:=\rho_S$ since the next theorem guarantees that the lcm period depends only on the matrix $S$ (does not depend on the vector $\bm{b}$ in the non-central case).
\begin{Them}[Kamiya, Takemura, and Terao \cite{KTT08, KTT11}]\label{thm-charquasi-KTT}
The function $|M_{\mathscr{A}_q}(q)|$ is a monic quasi-polynomial with a period $\rho_{\mathscr{A}}$ having the gcd property with respect to $\rho_{\mathscr{A}}$.
\end{Them}
The quasi-polynomial in Theorem \ref{thm-charquasi-KTT} is called the characteristic quasi-polynomial of the arrangement $\mathscr{A}$, and we write $\chi_{\mathscr{A}}^{quasi}(q)=|M_{\mathscr{A}_q}(q)|$.
The values of $k$-constituents are described as $\chi^k_{\mathscr{A}}(q)=|M_{\mathscr{A}_q}(q)|$ for all $k|\rho_{\mathscr{A}}$ and $q\in k+\rho_{\mathscr{A}}\mathbb{Z}$ with $q\gg 0$.
The next theorem states that we can obtain the minimum period by computing the lcm period in the central case\footnote{There exist examples of arrangements such that the lcm period is not the minimum period in the non-central case (see Theorem 1.1 in \cite{HTY2112})}.
\begin{Them}[Higashitani, Tran, and Yoshinaga \cite{HTY2112}]\label{thm-lcm=minimum}
The lcm period of $\chi^{quasi}_{\mathscr{A}}(q)$ coincides with the minimum period of $\chi^{quasi}_{\mathscr{A}}(q)$ in the central case.
\end{Them}

Kamiya, Takemura, and Terao computed the characteristic quasi-polynomials for Coxeter arrangements
\begin{align*}
\mathcal{A}_m&:=\left\{\{x_i-x_j=0\}\,\middle|\,1\leq i<j\leq m\right\},\\
\mathcal{B}_m&:=\left\{\{x_i=0\}\,\middle|\,1\leq i\leq m\right\}\cup\left\{\{x_i-x_j=0\},\{x_i+x_j=0\}\,\middle|\,1\leq i<j\leq m\right\},\\
\mathcal{C}_m&:=\left\{\{2x_i=0\}\,\middle|\,1\leq i\leq m\right\}\cup\left\{\{x_i-x_j=0\},\{x_i+x_j=0\}\,\middle|\,1\leq i<j\leq m\right\},\\
\mathcal{D}_m&:=\left\{\{x_i-x_j=0\},\{x_i+x_j=0\}\,\middle|\,1\leq i<j\leq m\right\}
\end{align*}
described by the orthonormal basis in section 4.5 in the paper \cite{KTT10-arxiv-v1} (the first version of the published paper \cite{KTT10}).
\begin{Them}[Kamiya, Takemura, and Terao \cite{KTT10-arxiv-v1}]\label{thm-cox-chrquaipoly-ONB}
The minimum periods for $\mathcal{A}_m$, $\mathcal{B}_m$, $\mathcal{C}_m$, and $\mathcal{D}_m$ are $\rho_{\mathcal{A}_m}=1, \rho_{\mathcal{B}_m}=\rho_{\mathcal{C}_m}=\rho_{\mathcal{D}_m}=2$.
Meanwhile the characteristic quasi-polynomials for $\mathcal{A}_m$, $\mathcal{B}_m$, $\mathcal{C}_m$, and $\mathcal{D}_m$ are as follows.
\begin{align}
&\chi_{\mathcal{A}_m}^{quasi}(q)=\prod_{i=1}^m (q-i+1).\label{eq-chi-quasi-coxA}\\
&\chi_{\mathcal{B}_m}^{quasi}(q)=
\begin{cases}
\prod_{i=1}^m(q-2i+1)\quad&\text{if}\ q\ \text{is odd},\\
(q-m)\prod_{i=1}^{m-1}(q-2i)\quad&\text{if}\ q\ \text{is even}.
\end{cases}\label{eq-chi-quasi-coxB}
\\
&\chi_{\mathcal{C}_m}^{quasi}(q)=
\begin{cases}
\prod_{i=1}^{m}(q-2i+1)\quad&\text{if}\ q\ \text{is odd},\\
\prod_{i=1}^{m}(q-2i)\quad&\text{if}\ q\ \text{is even}.
\end{cases}\label{eq-chi-quasi-coxC}
\\
&\chi_{\mathcal{D}_m}^{quasi}(q)=
\begin{cases}
(q-m+1)\prod_{i=1}^{m-1}(q-2i+1)\quad&\text{if}\ q\ \text{is odd},\\
\left(q^2-2(m-1)q+m(m-1)\right)\prod_{i=1}^{m-2}(q-2i)\quad&\text{if}\ q\ \text{is even}.
\end{cases}\label{eq-chi-quasi-coxD}
\end{align}
\end{Them}

\section{Characteristic quasi polynomials for deformation of Coxeter arrangements}\label{sec-char-poly-deform-Cox}
\subsection{Deformation of type A}
Let $t\in\mathbb{Z}_{\geq 0}$ with $0\leq t\leq m$, and let $s_1,\dots,s_t$ be positive integers with $s_t|\dots|s_1$.
We define a tuple $\bm{s}$ by $\bm{s}:=(s_1,\dots,s_t)$, and the arrangement $\mathcal{A}_m(\bm{s})$ by 
\begin{align*}
\mathcal{A}_m(\bm{s}):=\left\{\{s_ix_i=0\}\,\middle|\,1\leq i\leq t\right\}\cup\left\{\{x_i-x_j=0\}\,\middle|\,1\leq i<j\leq m\right\}.
\end{align*}
The corresponding matrix of $\mathcal{A}_m(\bm{s})$ is the $\displaystyle m\times\frac{m^2-m+2t}{2}$ matrix
\begin{align*}
\left(
\begin{array}{@{}cccc@{}}
s_1& \cdots&0&\\
\vdots  &\ddots&\vdots&\\
0&\cdots&s_t&A_m\\
\vdots  & &\vdots&\\
0&\cdots&0&
\end{array}
\right),
\end{align*}
where the matrix $A_m$ defines the arrangement $\mathcal{A}_m$.
We note that if $t=0$, then $\mathcal{A}_m(\bm{s})=\mathcal{A}_m$ and $\rho_{\mathcal{A}_m}=1$.
If $t\neq 0$, then $\rho_{\mathcal{A}_m(\bm{s})}=s_1$.
By Theorem \ref{thm-lcm=minimum}, the minimum period of $\chi^{quasi}_{\mathcal{A}_m(\bm{s})}(q)$ is
$
\begin{cases}
s_1\ &\text{if}\ t\neq 0,\\
1\ &\text{if}\ t=0.
\end{cases}
$ 

Let $k$ be a positive integer with $k|\rho_{\mathcal{A}_m(\bm{s})}$.
When $t\geq 1$, we define $d_i(k):=\gcd(k,s_i)$ for $1\leq i\leq t$.
We often omit $k$ and write $d_i=d_i(k)$ when $k$ is obvious.
For $q\in k+\rho_{\mathcal{A}_m(\bm{s})}\mathbb{Z}_{\geq 0}$, we have $d_i=\gcd(k,s_i)=\gcd(q,s_i)$  since $s_i|\rho_{\mathcal{A}_m(\bm{s})}$.
\begin{Lem}\label{lem-count-sx=0andx_i}
Let $t\geq 1$ and $k|\rho_{\mathcal{A}_m(\bm{s})}$. Let $q\in k+\rho_{\mathcal{A}_m(\bm{s})}\mathbb{Z}_{\geq 0}$ with $q>\max\{d_i(k)+i-1\mid 1\leq i\leq t\}$.
Let $h$ be an integer with $1\leq h\leq t$ and let $x_1,\dots,x_{h-1}\in\mathbb{Z}_q$ with
\begin{align}
s_ix_i\ne0 \quad&\text{for}\quad 1\leq i\leq h-1,\label{eq-count-sx=0}\\
x_i - x_j\ne0 \quad&\text{for}\quad 1\leq i<j \leq h-1.\label{eq-count-x_ineqx_j}
\end{align}
Then we have
\begin{align*}
\#\{x\in \mathbb{Z}_q\mid s_{h}x\neq 0\ \text{and}\ x\neq x_i\ \text{for any}\ 1\leq i\leq h-1\}=q-d_{h}(k)-h+1.
\end{align*}
\end{Lem}
\begin{proof}
The complement of $\{x\in \mathbb{Z}_q\mid s_{h}x=0\ \text{and}\ x\neq x_i\ \text{for any}\ 1\leq i\leq h-1\}$ is
\begin{align*}
\{x\in \mathbb{Z}_q\mid s_{h}x\neq 0\ \text{or}\ \text{there exists}\ i\ \text{with}\ 1\leq i\leq h-1\ \text{such that}\ x=x_i\}.
\end{align*}
By the inclusion--exclusion principle,
\begin{align*}
&\#\{x\in \mathbb{Z}_q\mid s_{h}x\neq 0\ \text{and}\ x\neq x_i\ \text{for any}\ 1\leq i\leq h-1\}\\
=&\,q-\#\{x\in \mathbb{Z}_q\mid s_{h}x=0\}-\#\{x\in \mathbb{Z}_q\mid 1\leq \exists i\leq h-1\ \text{s.t.}\ x=x_i\}\\
&\quad+\#\{x\in \mathbb{Z}_q\mid s_{h}x=0\ \text{and}\ 1\leq \exists i\leq h-1\ \text{s.t.}\ x=x_i\}.
\end{align*}
Since $\displaystyle \{x\in \mathbb{Z}_q\mid s_{h}x=0\}=\left\{0,\frac{q}{d_{h}},\frac{2q}{d_{h}},\ldots ,\frac{(d_{h}-1)q}{d_{h}}\right\}$, we have $\#\{x\in \mathbb{Z}_q\mid s_{h}x=0\}=d_{h}$.
In addition, clearly we have $\{x\in \mathbb{Z}_q\mid 1\leq \exists i\leq h-1\ \text{s.t.}\ x=x_i\}=\{x_1,\ldots ,x_{h-1}\}$.
By the condition \eqref{eq-count-x_ineqx_j}, the elements $x_1,\dots,x_h$ are different from each other.
So $\#\{x_1,\ldots ,x_{h-1}\}=h-1$.
To prove the assertion, it remains to prove that $\#\{x\in \mathbb{Z}_q\mid \mathbb{Z}_q\mid s_{h}x=0\ \text{and}\ 1\leq \exists i\leq h-1\ \text{s.t.}\ x=x_i\}=0$.
Assume that $\#\{x\in \mathbb{Z}_q\mid s_{h}x=0\ \text{and}\ 1\leq \exists i\leq h-1\ \text{s.t.}\ x=x_i\}\neq 0$.
Then there exists $i\leq h-1$ such that $s_{h}x_i=0$.
Since $s_h|s_i$, we can write $s_i=as_h$ for some $a\in\mathbb{Z}$.
Thus $s_ix_i=as_hx_i=0$.
This is in contradiction to the condition \eqref{eq-count-sx=0}.
Therefore we have $\#\{x\in \mathbb{Z}_q\mid s_{h}x\neq 0\ \text{and}\ x\neq x_i\ \text{for any}\ 1\leq i\leq h-1\}=q-d_{h}(k)-h+1$.
\end{proof}
We write $\displaystyle \prod_{i=a}^b f_i=1$ for integers $f_i\ (i\in\mathbb{Z})$ if $a>b$.
\begin{Them}\label{thm-charquaipoly-defomA}
Let $k|\rho_{\mathcal{A}_m(\bm{s})}$. Then we have
\begin{align}\label{eq-chi-quasi-Am}
\chi^k_{\mathcal{A}_m(\bm{s})}(q)=\prod_{i = 1}^{t}(q-d_i(k) - i+1)\prod_{i = t+1}^{m}(q-i+1).
\end{align}
\end{Them}
\begin{proof}
If $t=0$, then $\rho_{\mathcal{A}_m(\bm{s})}=1$. So $k=1$ and $q\in\mathbb{Z}_{>0}$.
The assertion follows from Theorem \ref{thm-cox-chrquaipoly-ONB}.

Let $t\geq 1$.
Let $q\in k+\rho_{\mathcal{A}_m(\bm{s})}\mathbb{Z}_{\geq 0}$ with $q>\max(\{m-1\}\cup\{d_i(k)+i-1\mid 1\leq i\leq t\})$.
The value $\chi^k_{A_m(\bm{s})}(q)=|M_{A_m(\bm{s})}(q)|$ is the number of elements $\bm{x}=(x_1,\ldots ,x_m)\in\mathbb{Z}_q^m$ with
\begin{align}
s_ix_i\ne0,\ &\text{for}\ 1\leq i\leq t,\label{eq-chi_Am-condi1}\\
x_i - x_j\ne0,\ &\text{for}\ 1\leq i < j \leq m.\label{eq-chi_Am-condi2}
\end{align}
We enumerate the elements $x_i\in\mathbb{Z}_q$ $(1\leq i\leq m)$ from the first entry $x_1$ to the last entry $x_m$ satisfying the conditions \eqref{eq-chi_Am-condi1} and \eqref{eq-chi_Am-condi2}.
We first consider $x_1$.
Since $\#\{x\in\mathbb{Z}_q\mid s_1x=0\}=\gcd(q,s_1)=d_1$, the number of elements $x_1\in\mathbb{Z}_q$ with $s_1x_1\neq 0$ is $q-d_1$.
Next let $2\leq h\leq t$, and we take $x_1,\dots,x_{h-1}\in\mathbb{Z}_q$ with
\begin{align*}
s_ix_i\ne0,\quad&\text{for}\quad 1\leq i\leq h-1,\\
x_i - x_j\ne 0\quad&\text{for}\quad 1\leq i < j \leq h-1.
\end{align*}
Then $x_1,\dots, x_h$ satisfy the conditions \eqref{eq-chi_Am-condi1} and \eqref{eq-chi_Am-condi2} if and only if $s_hx_h\neq 0$ and $x_h\neq x_l$ for any $1\leq l<h$.
By Lemma \ref{lem-count-sx=0andx_i}, after taking $x_1,\dots,x_{h-1}$, the number of way to take such $x_h\in\mathbb{Z}_q$ is $q-d_h-h+1$.
Therefore the number of elements $(x_1,\dots,x_t)\in\mathbb{Z}_q^{t}$ with the conditions \eqref{eq-chi_Am-condi1} and \eqref{eq-chi_Am-condi2} is
\begin{align*}
\prod_{i = 1}^{t}(q-d_i - i+1)=(q-d_1)(q-d_2-1)\cdots(q-d_t-t+1).
\end{align*}

Finally let $t<h\leq m$, and let us take $x_1,\dots,x_{h-1}$ satisfying the conditions \eqref{eq-chi_Am-condi1} and \eqref{eq-chi_Am-condi2}.
In this case, we can take any element avoiding $x_1,\dots.x_{h-1}$ as $x_h$.
So the number of elements $x_h\in\mathbb{Z}_q$ with the conditions \eqref{eq-chi_Am-condi1} and \eqref{eq-chi_Am-condi2} is $q-h+1$.
Hence we have $\chi^k_{A_m(\bm{s})}(q)=\prod_{i = 1}^{t}(q-d_i - i+1)\prod_{i = t+1}^{m}(q-i+1)$.
\end{proof}
\begin{Examp}
Let $t=0$. Then $\mathcal{A}_m(\bm{s})=\mathcal{A}_m$ and $\rho_{\mathcal{A}_m}=1$.
If $k|\rho_{\mathcal{A}_m}$, then $k=1$.
By substituting $t=0$ to the equality \eqref{eq-chi-quasi-Am}, we have
\begin{align*}
\chi^1_{\mathcal{A}_m}(q)=\prod_{i = 1}^{m}(q-i+1),
\end{align*}
and it coincides with the equality \eqref{eq-chi-quasi-coxA}.
\end{Examp}

\subsection{Deformation of type D}
Let $r,t\in\mathbb{Z}_{\geq 0}$ with $0\leq r\leq t\leq m$.
We take positive integers $s_1,\dots,s_t$, where $s_1,\dots,s_r$ are even, $s_{r+1},\dots,s_t$ are odd, and $s_t|\dots|s_r|\dots|s_1$.
Let $\bm{s}:=(s_1,\dots,s_t)$.
Then the arrangement $\mathcal{D}_m(\bm{s})$ is defined by 
\begin{align*}
\mathcal{D}_m(\bm{s}):=\left\{\{s_ix_i=0\}\,\middle|\,1\leq i\leq t\right\}\cup\left\{\{x_i-x_j=0\},\{x_i+x_j=0\}\,\middle|\,1\leq i<j\leq m\right\}.
\end{align*}
The corresponding matrix of $\mathcal{D}_m(\bm{s})$ is the $\displaystyle m\times(m^2-m+t)$ matrix
\begin{align*}
\left(
\begin{array}{@{}cccc@{}}
s_1& \cdots&0&\\
\vdots  &\ddots&\vdots&\\
0&\cdots&s_t&D_m\\
\vdots  & &\vdots&\\
0&\cdots&0&
\end{array}
\right),
\end{align*}
where the matrix $D_m$ defines the arrangement $\mathcal{D}_m$.
The class of arrangements represented by $\mathcal{D}_m(\bm{s})$ contains Coxeter arrangements $\mathcal{B}_m$, $\mathcal{C}_m$, and $\mathcal{D}_m$:
\begin{align*}
r=0,\ t=m,\ \text{and}\ s_i=1\ (i=1,\dots,t)\quad&\Rightarrow\quad \mathcal{D}_m(\bm{s})=\mathcal{B}_m,\\
r=t=m\ \text{and}\ s_i=2\ (i=1,\dots,t)\quad&\Rightarrow\quad \mathcal{D}_m(\bm{s})=\mathcal{C}_m,\\
r=t=0\quad&\Rightarrow\quad \mathcal{D}_m(\bm{s})=\mathcal{D}_m.
\end{align*}
By Theorem \ref{thm-lcm=minimum}, the minimum period of $\chi_{\mathcal{D}_m(\bm{s})}^{quasi}(q)$ is $\rho_{\mathcal{D}_m(\bm{s})}=
\begin{cases}
\lcm(s_1,2)\ &\text{if}\ t\neq 0,\\
2\ &\text{if}\ t=0.
\end{cases}$

Let $k$ be a positive integer with $k|\rho_{\mathcal{D}_m(\bm{s})}$ and let $q\in k+\rho_{\mathcal{D}_m(\bm{s})}\mathbb{Z}_{\ge0}$.
When $t\geq 1$, we define $d_i=d_i(k):=gcd(k,s_i)$ for $1\leq i\leq t$.
Then $d_i=gcd(k,s_i)=gcd(q,s_i)$ since $s_i|\rho_{\mathcal{D}_m(\bm{s})}$.
In addition, since $\rho_{\mathcal{D}_m(\bm{s})}$ is even, $k$ is even if and only if $q$ is even.
For $a,b\in\mathbb{Z}_q$, we denote $a\in\{b,-b\}$ by $a=\pm b$ for $a,b\in\mathbb{Z}_q$ while $a\not\in\{b,-b\}$ by $a\neq \pm b$.
\begin{Lem}\label{lem-count-sx=0andx_i-D}
Let $t\geq 1$ and $k|\rho_{\mathcal{D}_m(\bm{s})}$.
Let $q\in k+\rho_{\mathcal{D}_m(\bm{s})}\mathbb{Z}_{\geq 0}$ with $q>\max\{d_i(k)+2i-2\mid 1\leq i\leq t\}$.
Let $h$ be an integer with $1\leq h\leq t$ and let $x_1,\dots,x_{h-1}\in\mathbb{Z}_q$ with
\begin{align}
s_ix_i\ne0 \quad&\text{for}\quad 1\leq i\leq h-1,\label{eq-count-sx=0-D}\\
x_i - x_j\ne0 \quad&\text{for}\quad 1\leq i<j \leq h-1,\label{eq-count-x_ineqx_j-D}\\
x_i + x_j\ne0 \quad&\text{for}\quad 1\leq i<j \leq h-1.\label{eq-count-x_i+x_jneq-D}
\end{align}
In addition, if $k$ is even and $r+1\leq h-1$, then we assume that $\displaystyle x_i\neq q/2$ for $r+1\leq i\leq h-1$.
Then we have
\begin{align*}
\#\{x\in \mathbb{Z}_q\mid s_{h}x\neq 0\ \text{and}\ x\neq \pm x_i\ \text{for}\ 1\leq i\leq h-1\}=q-d_{h}(k)-2h+2.
\end{align*}
\end{Lem}
\begin{proof}
By the inclusion--exclusion principle,
\begin{align*}
&\#\{x\in \mathbb{Z}_q\mid s_{h}x\neq 0\ \text{and}\ x\neq \pm x_i\ \text{for}\ 1\leq i\leq h-1\}\\
=&\,q-\#\{x\in \mathbb{Z}_q\mid s_{h}x= 0\ \text{or}\ 1\leq \exists i\leq h-1 x= \pm x_i\}\\
=&\,q-\#\{x\in \mathbb{Z}_q\mid s_{h}x=0\}-\#\{x\in \mathbb{Z}_q\mid 1\leq \exists i\leq h-1\ \text{s.t.}\ x= \pm x_i\}\\
&\quad +\#\{x\in \mathbb{Z}_q\mid s_{h}x=0\ \text{and}\ 1\leq \exists i\leq h-1\ \text{s.t.}\ x= \pm x_i\}.
\end{align*}
Here we have $\#\{x\in \mathbb{Z}_q\mid s_{h}x=0\}=d_h$ and $\#\{x\in \mathbb{Z}_q\mid s_{h}x=0\ \text{and}\ 1\leq \exists i\leq h-1\ \text{s.t.}\ x= \pm x_i\}=0$ similarly to Lemma \ref{lem-count-sx=0andx_i}.
It remains to prove the assertion that $\#\{x\in \mathbb{Z}_q\mid s_{h}x=0\ \text{and}\ 1\leq \exists i\leq h-1\ \text{s.t.}\ x= \pm x_i\}=\#\{x_i,-x_i\mid 1\leq i\leq h-1\}=2h-2$.
By the conditions \eqref{eq-count-x_ineqx_j-D} and \eqref{eq-count-x_i+x_jneq-D}, to prove $\#\{x_i,-x_i\mid 1\leq i\leq h-1\}=2h-2$, we may prove that $x_i\neq -x_i$ for $1\leq i\leq h-1$. Let $1\leq i\leq h-1$.
\begin{itemize}
\item When $k$ is odd ($\Leftrightarrow$ $q$ is odd), it holds that $\{x\in\mathbb{Z}_q\mid x=-x\}=\{0\}$.
By the condition \eqref{eq-count-sx=0-D}, we have $x_i\neq 0$. Hence $x_i\neq -x_i$.
\item When $k$ is even ($\Leftrightarrow$ $q$ is even) and $i\leq r$, it holds that $\displaystyle \{x\in\mathbb{Z}_q\mid x=-x\}=\left\{0,q/2\right\}$ while $s_i$ is even.
If $x_i=0$, then $s_ix_i=0$.
If $\displaystyle x_i=q/2$, then $\displaystyle s_ix_i=q\frac{s_i}{2}=0$.
These are in contradiction to the condition \eqref{eq-count-sx=0-D}. Therefore $x_i\neq -x_i$.
\item When $k$ is even ($\Leftrightarrow$ $q$ is even) and $r+1\leq i\leq h-1$, we have $x_i\neq -x_i$ by the assumption $\displaystyle x_i\neq q/2$ for $r+1\leq i\leq h-1$ and the condition \eqref{eq-count-sx=0-D}.
\end{itemize}
Thus we have $\#\{x_i,\ -x_i\mid 1\leq i\leq h-1\}=2h-2$, and hence $\#\{x\in \mathbb{Z}_q\mid s_{h}x\neq 0\ \text{and}\ x\neq \pm x_i\ \text{for}\ 1\leq i\leq h-1\}=q-d_h-2h+2$.
\end{proof}
We write $\displaystyle \sum_{i=a}^b f_i=0$ and $\displaystyle \prod_{i=a}^b f_i=1$ for integers $f_i\ (i\in\mathbb{Z})$ if $a>b$.
\begin{Them}\label{thm-charquaipoly-defomD}
Let $k|\rho_{\mathcal{D}_m(\bm{s})}$.
If $k$ is odd, then
\begin{align}\label{eq-chi-quasi-Dm-odd}
\chi^k_{\mathcal{D}_m(\bm{s})}(q)=\prod_{i = 1}^{t}(q-d_i(k)-2i+2)\left(\prod_{i = t+1}^m(q-2i+1)+(m-t)\prod_{i= t+1}^{m-1}(q-2i+1)\right).
\end{align}
If $k$ is even, then
\begin{align}\label{eq-chi-quasi-Dm-even}
\chi^k_{\mathcal{D}_m(\bm{s})}(q)=\prod_{i=1}^{r}(q-d_i(k)-2i+2)(P_1P_2+ P_3P_4),
\end{align}
where
\begin{align*}
P_1&:=\prod_{i=r+1}^{t}(q-d_i(k)-2i+1),\\
P_2&:=\prod_{i=t+1}^{m}(q-2i)+2(m-t)\prod_{i=t+1}^{m-1}(q-2i)+(m-t)(m-t-1)\prod_{i=t+1}^{m-2}(q-2i),\\
P_3&:=\sum_{i=r+1}^{t}\prod_{j=1}^{i-1}(q-d_j(k)-2j+1)\prod_{j=i+1}^{t}(q-d_j(k)-2j+3),\\
P_4&:=\prod_{i=t+1}^{m}(q-2i+2)+(m-t)\prod_{i=t+1}^{m-1}(q-2i+2).
\end{align*}

In particular, when $t=m$ and $\displaystyle s_{r+1}=s_{r+2}=\cdots=s_m$, the $k$-constituents are described as follows.
If $k$ is odd, then
\begin{align*}
\chi_{\mathcal{D}_{m}(\bm{s})}^k(q)=\prod_{i=1}^{r}(q-d_i(k)-2i+2)\prod_{i=r+1}^{m}(q-d(k)-2i+2).
\end{align*}
If $k$ is even, then
\begin{align*}
\chi_{\mathcal{D}_{m}(\bm{s})}^k(q)&=
\begin{cases}
\prod_{i=1}^{m}(q-d_i(k)-2i+2)&\ (r=m),\\
(q-d(k)-m-r+1)\prod_{i=1}^{r}(q-d_i(k)-2i+2)\prod_{i=r+1}^{m-1}(q-d(k)-2i+1)&\ (r<m),
\end{cases}
\end{align*}
where we write $d(k):=d_{r+1}(k)=d_{r+2}(k)=\cdots=d_{m}(k)$
\end{Them}
\begin{proof}
Let $q\in k+\rho_{\mathcal{D}_m(\bm{s})}\mathbb{Z}_{\geq 0}$ with $q>\max(\{2m\}\cup\{d_i(k)+2i-1\mid 1\leq i\leq t\})$.
The value $\chi_{\mathcal{D}_{m}(\bm{s})}^k(q)$ is the number of elements $\bm{x}=(x_1,\dots,x_m)\in\mathbb{Z}_q^m$ with
\begin{align}
s_ix_i\ne0 \quad&\text{for}\ 1\leq i\leq t,\label{eq-condi-M(A)-D1}\\
x_i - x_j\ne0 \quad&\text{for}\ 1\leq i<j \leq m,\label{eq-condi-M(A)-D2}\\
x_i + x_j\ne0 \quad&\text{for}\ 1\leq i<j \leq m.\label{eq-condi-M(A)-D3}
\end{align}
To prove the assertion, we enumerate the elements $x_i\in\mathbb{Z}_q$ $(1\leq i\leq m)$ with the conditions \eqref{eq-condi-M(A)-D1}, \eqref{eq-condi-M(A)-D2}, and \eqref{eq-condi-M(A)-D3} from the first entry $x_1$ to the last entry $x_m$.

First let $k$ be odd.
By Lemma \ref{lem-count-sx=0andx_i-D}, the number of elements $(x_1,\dots,x_t)\in\mathbb{Z}_q^t$ with the conditions \eqref{eq-condi-M(A)-D1}, \eqref{eq-condi-M(A)-D2}, and \eqref{eq-condi-M(A)-D3} is
\begin{align*}
(q-d_1)(q-d_2-2)\cdots(q-d_t-2t+2)=\prod_{i=1}^t(q-d_i-2i+2).
\end{align*}
Here we fix $(x_1,\dots,x_t)$.
The condition \eqref{eq-condi-M(A)-D1} is irrelevant for $x_{t+1},\dots,x_m$.
We determine the number of elements $(x_{t+1},\dots,x_m)\in\mathbb{Z}_q^{m-t}$ with the conditions \eqref{eq-condi-M(A)-D2} and \eqref{eq-condi-M(A)-D3} by dividing the following cases.
\begin{itemize}
\item When $x_i\neq 0$ for all $t+1\leq i\leq m$, the number of $(x_{t+1},\dots,x_m)\in\mathbb{Z}_q^{m-t}$ with the conditions \eqref{eq-condi-M(A)-D2} and \eqref{eq-condi-M(A)-D3} is
$\displaystyle (q-1-2t)(q-1-2t+2)\cdots(q-1-2m+2)=\prod_{i=t+1}^m(q-2i+1)$.
\item When there exists $t+1\leq i\leq m$ such that $x_i= 0$, such $i$ is unique by the condition \eqref{eq-condi-M(A)-D2} or \eqref{eq-condi-M(A)-D3}.
We fix an index $i$ with $x_i=0$.
Then the number of elements $(x_{t+1},\dots,x_{i-1},x_{i+1},\dots,x_m)\in\mathbb{Z}_q^{m-t-1}$ with the conditions \eqref{eq-condi-M(A)-D2} and \eqref{eq-condi-M(A)-D3} is $\displaystyle (q-1-2t)(q-1-2t+2)\cdots(q-1-2m+4)$.
Since there are $(m-t)$ ways to select such index $i$, the number of $(x_{t+1},\dots,x_m)\in\mathbb{Z}_q^{m-t}$ with the conditions\eqref{eq-condi-M(A)-D2} and \eqref{eq-condi-M(A)-D3} is $\displaystyle (m-t)(q-1-2t)(q-1-2t+2)\cdots(q-1-2m+4)=(m-t)\prod_{i=t+1}^{m-1}(q-2i+1)$.
\end{itemize}
From the argument above, we have
\begin{align*}
\chi^k_{\mathcal{D}_m(\bm{s})}(q)=\prod_{i = 1}^{t}(q-d_i-2i+2)\left(\prod_{i = t+1}^m(q-2i+1)+(m-t)\prod_{i= t+1}^{m-1}(q-2i+1)\right).
\end{align*}

Next let $k$ be even. We recall that $q$ is even in this case.
By Lemma \ref{lem-count-sx=0andx_i-D}, the number of elements $(x_1,\dots,x_r)\in\mathbb{Z}_q^r$ with the conditions \eqref{eq-condi-M(A)-D1}, \eqref{eq-condi-M(A)-D2}, and \eqref{eq-condi-M(A)-D3} is 
\begin{align*}
(q-d_1)(q-d_2-2)\cdots(q-d_r-2r+2)=\prod_{i=1}^r(q-d_i-2i+2).
\end{align*}
We fix $(x_1,\dots,x_r)$.
Then we enumerate elements $(x_{r+1},\dots,x_m)\in\mathbb{Z}_q^{m-r}$ with the conditions \eqref{eq-condi-M(A)-D1}, \eqref{eq-condi-M(A)-D2}, and \eqref{eq-condi-M(A)-D3} by dividing several cases.
We note that $\displaystyle x_1\not\in\left\{0,q/2\right\},\dots,x_r\not\in\left\{0,q/2\right\}$ and $x_{r+1}\neq 0,\dots,x_{t}\neq 0$.

(i) We consider the case when $\displaystyle q/2$ is not contained in $\{x_{r+1},\dots,x_{t}\}$.
For $r+1\leq i\leq t$, there are $q-1-d_i$ ways to take $x_i\in\mathbb{Z}_q$ with $s_ix_i\neq 0$ and $x_i\neq q/2$.
In addition, the conditions \eqref{eq-condi-M(A)-D2} and \eqref{eq-condi-M(A)-D3} prohibits $2i-2$ ways to take $x_i\in\mathbb{Z}_q$.
So the number of elements $(x_{r+1},\dots,x_{t})$ with the conditions \eqref{eq-condi-M(A)-D1}, \eqref{eq-condi-M(A)-D2}, and \eqref{eq-condi-M(A)-D3} is
\begin{align*}
(q-1-d_{r+1}-2r)(q-1-d_{r+2}-2r-2)\cdots(q-1-d_t-2t+2)=\prod_{i=r+1}^t(q-d_i-2i+1).
\end{align*}
Then after fixing $(x_{r+1},\dots,x_t)$, we determine the number of elements $(x_{t+1},\dots,x_{m})$ with the conditions \eqref{eq-condi-M(A)-D2} and \eqref{eq-condi-M(A)-D3} since the condition \eqref{eq-condi-M(A)-D1} is irrelevant for $x_{t+1},\dots,x_m$.
\begin{itemize}
\item[(i-a)] When both $0$ and $\displaystyle q/2$ are not contained in $\{x_{t+1},\dots,x_m\}$, we can take $x_i$ from $\displaystyle \mathbb{Z}_q\setminus\left\{0,q/2,\pm x_1,\dots,\pm x_{i-1}\right\}$ for $t+1\leq i\leq m$. Therefore the number of elements $(x_{t+1},\dots,x_{m})$ with the conditions \eqref{eq-condi-M(A)-D2} and \eqref{eq-condi-M(A)-D3} is $\displaystyle \prod_{i=t+1}^m(q-2-2i+2)=\prod_{i=t+1}^m(q-2i)$.
\item[(i-b)] When exactly one of $\displaystyle\{0,q/2\}$ is contained in $\{x_{t+1},\dots,x_m\}$, we select an index $j$ from $\{t+1,\dots,m\}$ so that $\displaystyle x_j\in\left\{0,q/2\right\}$.
There are $(m-t)$ ways to select such index $j$.
We may assume that $\displaystyle x_m\in\left\{0,q/2\right\}$ after selecting such index.
For $t+1\leq i\leq m-1$, the $i$-th entry $x_{i}$ can be taken from $\displaystyle \mathbb{Z}_q\setminus\left\{0,q/2,\pm x_1,\dots,\pm x_{i-1}\right\}$.
Hence the number of elements $(x_{t+1},\dots,x_{m})$ with the conditions \eqref{eq-condi-M(A)-D2} and \eqref{eq-condi-M(A)-D3} is $\displaystyle 2(m-t)\prod_{i=t+1}^{m-1}(q-2-2i+2)=2(m-t)\prod_{i=t+1}^{m-1}(q-2i)$.
\item[(i-c)] When both $0$ and $\displaystyle q/2$ are contained in $\{x_{t+1},\dots,x_m\}$, there are $(m-t)(m-t-1)$ ways to select indices $j,l$ with $\displaystyle x_j=0,\ x_l=q/2$.
Similarly to the case (i-b), the number of elements $(x_{t+1},\dots,x_{m})$ with the conditions \eqref{eq-condi-M(A)-D2} and \eqref{eq-condi-M(A)-D3} is $\displaystyle (m-t)(m-t-1)\prod_{i=t+1}^{m-2}(q-2-2i+2)=(m-t)(m-t-1)\prod_{i=t+1}^{m-2}(q-2i)$.
\end{itemize}
Hence in the case (i), the number of elements $(x_{r+1},\dots,x_{m})$ with the conditions \eqref{eq-condi-M(A)-D1}, \eqref{eq-condi-M(A)-D2}, and \eqref{eq-condi-M(A)-D3} is $P_1P_2$.

(i\hspace{-0.5mm}i) When $\displaystyle q/2$ is contained in $\{x_{r+1},\dots,x_{t}\}$, the index $i$ with $\displaystyle x_i=q/2$ is unique by the conditions \eqref{eq-condi-M(A)-D2} and \eqref{eq-condi-M(A)-D3}.
If $r+1\leq j\leq i-1$, then we can take $\displaystyle x_j\in \mathbb{Z}_q\setminus\left\{q/2\right\}$ with $\displaystyle s_jx_j\neq 0$, $x_j\neq \pm x_1,\ \dots,\ x_{j}\neq\pm x_{j-1}$.
If $i+1\leq j\leq t$, then we can take $\displaystyle x_j\in\mathbb{Z}_q\setminus\left\{q/2\right\}$ with $s_jx_j\neq 0$, $x_j\neq \pm x_1$, $\ldots$, $x_j\neq \pm x_{i-1}$, $x_j\neq \pm x_{i+1}$, $\ldots$, $x_{j}\neq\pm x_{j-1}$.
Therefore by Lemma \ref{lem-count-sx=0andx_i-D}, the number of elements $(x_{r+1},\dots,x_{t})$ with the conditions \eqref{eq-condi-M(A)-D1}, \eqref{eq-condi-M(A)-D2}, and \eqref{eq-condi-M(A)-D3} is
\begin{align*}
&\sum_{i=r+1}^{t}\prod_{j=r+1}^{i-1}(q-1-d_j-2j+2)\prod_{j=i+1}^{t}(q-1-d_j-2j+4)\\
=\,&\sum_{i=r+1}^{t}\prod_{j=r+1}^{i-1}(q-d_j-2j+1)\prod_{j=i+1}^{t}(q-d_j-2j+3).
\end{align*}
Furthermore after fixing $(x_{r+1},\dots,x_t)$, we enumerate elements $(x_{t+1},\dots,x_{m})$ with the conditions \eqref{eq-condi-M(A)-D2} and \eqref{eq-condi-M(A)-D3}.
We note that $\displaystyle x_{t+1}\neq q/2,\dots,x_m\neq q/2$ since $\displaystyle q/2\in\left\{x_{r+1},\dots,x_{t}\right\}$.
Let us fix the index $r+1\leq l\leq t$ with $\displaystyle x_l=q/2$.
\begin{itemize}
\item[(i\hspace{-0.5mm}i-a)] When $0\not\in\{x_{t+1},\dots,x_m\}$, we can take $x_i$ from $\displaystyle \mathbb{Z}_q\setminus\{0,q/2,\pm x_1,\dots,\pm x_{l-1},\pm x_{l+1},\dots,\pm x_{i-1}\}$ for $t+1\leq i\leq m$.
Therefore the number of elements $(x_{t+1},\dots,x_{m})$ with the conditions \eqref{eq-condi-M(A)-D2} and \eqref{eq-condi-M(A)-D3} is $\displaystyle \prod_{i=t+1}^m(q-2-2i+4)=\prod_{i=t+1}^m(q-2i+2)$.
\item[(i\hspace{-0.5mm}i-b)] When $0\in\{x_{t+1},\dots,x_m\}$, there are $(m-t)$ ways to select an index $t+1\leq j\leq m$ with $x_j=0$.
We may assume that $\displaystyle x_m=0$.
Then we can take $x_i$ from $\displaystyle \mathbb{Z}_q\setminus\{0,q/2,\pm x_1,\dots,\pm x_{l-1},\pm x_{l+1},\dots,\pm x_{i-1}\}$ for $t+1\leq i\leq m-1$.
Therefore the number of elements $(x_{t+1},\dots,x_{m})$ with the conditions \eqref{eq-condi-M(A)-D2} and \eqref{eq-condi-M(A)-D3} is $\displaystyle (m-t)\prod_{i=t+1}^{m-1}(q-2-2i+4)=(m-t)\prod_{i=t+1}^{m-1}(q-2i+2)$.
\end{itemize}
In the case (i\hspace{-0.5mm}i), the number of elements $(x_{r+1},\dots,x_{m})$ with the conditions \eqref{eq-condi-M(A)-D1}, \eqref{eq-condi-M(A)-D2}, and \eqref{eq-condi-M(A)-D3} is $P_3P_4$.
Hence if $k$ is even, then we have
\begin{align*}
\chi^k_{\mathcal{D}_m(\bm{s})}(q)=\prod_{i=1}^{r}(q-d_i-2i+2)(P_1P_2+ P_3P_4).
\end{align*}

Finally, we assume $t=m$ and we write $d:=d_{r+1}=d_{r+2}=\cdots=d_m$.
If $k$ is odd, then clearly $\chi_{D_{m}(\bm{s})}^k(q)=\prod_{i=1}^{r}(q-d_i-2i+2)\prod_{i=r+1}^{m}(q-d-2i+2)$.
If $k$ is even, then
\begin{align*}
&P_1P_2+ P_3P_4=P_1+P_3\\
=&\,\prod_{i=r+1}^{m}(q-d-2i+1)+\sum_{i=r+1}^{m}\prod_{j=1}^{i-1}(q-d-2j+1)\prod_{j=i+1}^{m}(q-d-2j+3)\\
=&\,\prod_{i=r+1}^{m}(q-d-2i+1)+\sum_{i=r+1}^{m}\prod_{j=1}^{m-1}(q-d-2j+1)\\
=&\,\prod_{i=r+1}^{m}(q-d-2i+1)+(m-r)\prod_{i=1}^{m-1}(q-d-2i+1)\\
=&\,
\begin{cases}
1&\ (r=m),\\
(q-d-m-r+1)\prod_{i=r+1}^{m-1}(q-d-2i+1)&\ (r<m).
\end{cases}
\end{align*}
\end{proof}
All $k$-constituents for $\chi^{quasi}_{\mathcal{D}_m(\bm{s})}(q)$ are factorized if $t=m$ and $\displaystyle s_{r+1}=s_{r+2}=\cdots=s_m$.
In addition. the characteristic quasi-polynomials for $\mathcal{B}_m$, $\mathcal{C}_m$, and $\mathcal{D}_m$ are computed from Theorem \ref{thm-charquaipoly-defomD}.
\begin{Examp}
$(1)$ Let $r=0$, $t=m$, and $s_1=\cdots=s_m=1$.
We have that $\mathcal{D}_m(\bm{s})=\mathcal{B}_m$, $\rho_{\mathcal{B}_m}=\lcm(s_1,2)=2$, and $d_1=\cdots=d_m=1$. Then
\begin{align*}
\chi^1_{B_m}(q)&=\prod_{i = 1}^{m}(q-1-2i+2)=\prod_{i = 1}^{m}(q-2i+1),\\
\chi^2_{B_m}(q)&=(q-1-m-0+1)\prod_{i = 1}^{m-1}(q-1-2i+1)=(q-m)\prod_{i = 1}^{m-1}(q-2i).
\end{align*}

$(2)$ Let $r=t=m$, and $s_1=\cdots=s_m=2$.
We have $\mathcal{D}_m(\bm{s})=\mathcal{C}_m$ and $\rho_{\mathcal{C}_m}=\lcm(s_1,2)=2$.
If $k=1$, then $d_1=\cdots=d_m=1$ and if $k=2$, then $d_1=\cdots=d_m=2$.
So we have
\begin{align*}
\chi^1_{C_m}(q)&=\prod_{i = 1}^{m}(q-1-2i+2)=\prod_{i = 1}^{m}(q-2i+1).\\
\chi^2_{C_m}(q)&=\prod_{i = 1}^{m}(q-2-2i+2)=\prod_{i = 1}^{m}(q-2i).
\end{align*}

$(3)$ Let $r=t=0$.
We have $\mathcal{D}_m(\bm{s})=\mathcal{D}_m$ and $\rho_{\mathcal{D}_m}=2$.
Then
\begin{align*}
\chi^1_{D_m}(q)&=\prod_{i=1}^m(q-2i+1)+m\prod_{i=1}^{m-1}(q-2i+1)=(q-m+1)\prod_{i=1}^{m-1}(q-2i+1),\\
\chi^2_{D_m}(q)&=\prod_{i=1}^{m}(q-2i)+2m\prod_{i=1}^{m-1}(q-2i)+m(m-1)\prod_{i=1}^{m-2}(q-2i)\\
&=\left((q-2m)(q-2m+2)+2m(q-2m+2)+m(m-1)\right)\prod_{i=1}^{m-2}(q-2i)\\
&=\left(q^2-2(m-1)q+m(m-1)\right)\prod_{i=1}^{m-2}(q-2i).
\end{align*}
These quasi-polynomials coincide with that of \eqref{eq-chi-quasi-coxB}, \eqref{eq-chi-quasi-coxC}, and \eqref{eq-chi-quasi-coxD}.
\end{Examp}

\section{Deleting hyperplanes}\label{sec-deleting-hyperplane}
In this section, we introduce a sequence of arrangements that can be obtained by deleting the hyperplanes while keeping the condition that all $k$-constituents are factorized.
For $H_q\in\mathscr{A}_q$, we call $\mathscr{A}_q\setminus\{H_q\}$ the deletion and we call
\begin{align*}
\mathscr{A}_q^{H_q}:=\left\{H_q\cap H_{q}^{\prime}\,\middle|\,H_{q}^{\prime}\in\mathscr{A}_q\setminus\{H_q\},\ H_q\cap H_{q}^{\prime}\neq\emptyset\right\}
\end{align*}
the restriction in $\mathbb{Z}/q\mathbb{Z}$.

\begin{Lem}\label{lem-inc-exc}
Let $\mathscr{A}$ be an arrangement such that all coefficients of the defining polynomials of all hyperplanes are integers.
Let $H\in\mathscr{A}$. For any $q\in\mathbb{Z}_{>0}$,
\begin{align*}
\left|M_{\mathscr{A}_q}(q)\right|=\left|M_{\mathscr{A}_q\setminus\{H_q\}}(q)\right|-\left|H_q\right|+\left|\bigcup_{H_{q}^{\prime}\in\mathscr{A}_q\setminus\{H_q\}}H_q\cap H_{q}^{\prime}\right|.
\end{align*}
\end{Lem}
\begin{proof}
It holds that $\displaystyle\bigcup_{H_{q}^{\prime}\in\mathscr{A}_q}H_q^{\prime}=\left(\bigcup_{H_{q}^{\prime}\in\mathscr{A}_q\setminus\{H_q\}}H_q^{\prime}\right)\cup H_q$.
By the inclusion-exclusion principle, we have
\begin{align*}
\left|\bigcup_{H_{q}^{\prime}\in\mathscr{A}_q}H_q^{\prime}\right|=\left|\bigcup_{H_{q}^{\prime}\in\mathscr{A}_q\setminus\{H_q\}}H_q^{\prime}\right| + \left|H_q\right| - \left|\bigcup_{H_{q}^{\prime}\in\mathscr{A}_q\setminus\{H_q\}}H_q\cap H_q^{\prime}\right|.
\end{align*}
Therefore
\begin{align*}
\left|\mathbb{Z}_q^m\right|-\left|\bigcup_{H_{q}^{\prime}\in\mathscr{A}_q}H_q^{\prime}\right|=\left|\mathbb{Z}_q^m\right|-\left|\bigcup_{H_{q}^{\prime}\in\mathscr{A}_q\setminus\{H_q\}}H_q^{\prime}\right| - \left|H_q\right| + \left|\bigcup_{H_{q}^{\prime}\in\mathscr{A}_q\setminus\{H_q\}}H_q\cap H_q^{\prime}\right|.
\end{align*}
This means $\left|M_{\mathscr{A}_q}(q)\right|=\left|M_{\mathscr{A}_q\setminus\{H_q\}}(q)\right|-\left|H_q\right|+\left|\bigcup_{H_{q}^{\prime}\in\mathscr{A}_q\setminus\{H_q\}}H_q\cap H_{q}^{\prime}\right|$.
\end{proof}
A matrix $P\in\Mat_{m\times m}(\mathbb{Z})$ is said to be unimodular if $\det(P)\in\{1,-1\}$.
If $P\in\Mat_{m\times m}(\mathbb{Z})$ is unimodular, then the inverse $P^{-1}\in\Mat_{m\times m}(\mathbb{Z})$ is also unimodular.
\begin{Cor}\label{Cro-del-rest-gcd=1}
Let $\mathscr{A}$ be an arrangement such that all coefficients of the defining polynomials of all hyperplanes are integers.
Let $\rho$ be a period of $\mathscr{A}$ and let $H\in\mathscr{A}$.
We write $H=\{a_1x_1+\cdots+a_mx_m=b\}$ for some $a_1,\dots,a_m,b\in\mathbb{Z}$.
If there exists a unimodular matrix $P\in\Mat_{m\times m}(\mathbb{Z})$ such that $P\,^t([a_1]_q,\dots,[a_m]_q)=\,^t([1]_q,[0]_q,\dots,[0]_q)$, then the deletion-restriction formula
\begin{align*}
\left|M_{\mathscr{A}_q}(q)\right|=\left|M_{\mathscr{A}_q\setminus\{H_q\}}(q)\right|-\left|M_{\mathscr{A}_q^{H_q}}(q)\right|
\end{align*}
holds for any $q\in\mathbb{Z}_{>0}$.
\end{Cor}
\begin{proof}
We set $[h]_q:=\,^t([a_1]_q,\dots,[a_m]_q)$.
Then we have
\begin{align*}
\left|H_q\right|&=\left|\left\{\bm{x}\in\mathbb{Z}_q^m\,\middle|\,\bm{x}[h]_q=[b]_q\right\}\right|
=\left|\left\{\bm{x}\in\mathbb{Z}_q^m\,\middle|\,\bm{x}P^{-1}P[h]_q=[b]_q\right\}\right|\\
&=\left|\left\{\bm{y}\in\mathbb{Z}_q^m\,\middle|\,[y_1]_q=[b]_q\right\}\right|
=q^{m-1}=\left|\mathbb{Z}_q^{m-1}\right|,
\end{align*}
where we write $\bm{y}=\bm{x}P^{-1}$ in the equalities above.
So we have $\left|H_q\right| - \left|\bigcup_{H_{q}^{\prime}\in\mathscr{A}_q\setminus\{H_q\}}H_q\cap H_q^{\prime}\right|=\left|M_{\mathscr{A}_q^{H_q}}(q)\right|$.
\end{proof}
Corollary \ref{Cro-del-rest-gcd=1} is the special case of Corollary 4.11 in the paper \cite{Liu-Tran-Yoshinaga21}.
If $\gcd(a_1,\dots,a_m)=1$, then the assumption (existence of a unimodular matrix $P$) of Lemma \ref{Cro-del-rest-gcd=1} holds.
In addition, Corollary \ref{Cro-del-rest-gcd=1} implies that the deletion-restriction formula for characteristic quasi-polynomials holds if there exists an index $i$ such that the coefficient of $x_i$ is $1$.
If there are no variables in $H$ such that the coefficient is $1$, then the row operation corresponding to the product of unimodular matrices makes it easier to understand the restriction arrangement.
For example, let $\mathscr{A}=\{\{3x_1+2x_2+x_3=0\},\{2x_1+5x_2=0\},\{2x_1+4x_2+2x_3=0\}\}$ and $H=\{2x_1+5x_2=0\}$. Since
\begin{align*}
\left(
\begin{matrix}
3&2&2\\ 2&5&4\\ 1&0&2
\end{matrix}
\right)\ \rightarrow\ \left(
\begin{matrix}
3&2&2\\ -4&1&0\\ 1&0&2
\end{matrix}
\right)\ \rightarrow\ \left(
\begin{matrix}
11&0&2\\ -4&1&0\\ 1&0&2
\end{matrix}
\right)\ \rightarrow\ \left(
\begin{matrix}
-4&1&0\\ 11&0&2\\ 1&0&2
\end{matrix}
\right),
\end{align*}
we have $\mathscr{A}_q^{H_q}=\{\{11x_2+x_3=0\}_q,\{2x_2+2x_3=0\}_q\}$ for $q\gg 0$.

To emphasize which dimension of arrangements is considered in the following subsections, an $m$-dimensional hyperplane arrangement is sometimes referred to as an “$m$-arrangement.”

\subsection{A sequence of subarrangements of deformation of type A}
We give a sequence of subarrangements of $\mathcal{A}_m(\bm{s_m})$ that all $k$-constituents are factorized.
We recall that
\begin{align*}
\mathcal{A}_m(\bm{s_m})=\left\{\{s_ix_i=0\}\,\middle|\,1\leq i\leq m\right\}\cup\left\{\{x_i-x_j=0\}\,\middle|\,1\leq i<j\leq m\right\}.
\end{align*}
\begin{Them}\label{thm-deleting-typeA}
Let $\bm{s}_m=(s_1,\dots,s_m)$ be a tuple of integers with $s_m|\cdots|s_1$.
Let $\mathscr{B}$ be a subset of $\{\{x_i-x_m=0\}\mid 1\leq i\leq m-1\}$ with $|\mathscr{B}|=l$.
Let
\begin{align*}
\mathscr{A}(\mathscr{B})=\mathcal{A}_m(\bm{s}_m)\setminus\left(\{\{s_mx_m=0\}\}\cup\mathscr{B}\right)\quad(0\leq l\leq m-1).
\end{align*}
Then $\rho_{\mathscr{A}(\mathscr{B})}=s_1$ and the $k$-constituent of $\chi_{\mathscr{A}(\mathscr{B})}^{quasi}(q)$ is
\begin{align}\label{eq-chiAmsm-deleting}
\chi_{\mathscr{A}(\mathscr{B})}^{k}(q)=(q-m+l+1)\prod_{i=1}^{m-1}(q-d_i(k)-i+1)
\end{align}
for $k|s_1$.
\end{Them}
\begin{proof}
Since $\rho_{\mathcal{A}_m(\bm{s}_m)}=s_1$, we also have $\rho_{\mathscr{A}(\mathscr{B})}=s_1$.
Let $q\in k+\rho_{\mathcal{A}_m(\bm{s}_m)}\mathbb{Z}$ with $q\gg 0$.
We prove the equality \eqref{eq-chiAmsm-deleting} by induction on $l$.
First let $l=0$, that is, $\mathscr{B}=\emptyset$.
We define the subarrangement $\overline{\mathcal{A}}$ of the $m$-arrangement $\mathcal{A}_m(\bm{s}_m)$ by
\begin{align*}
\overline{\mathcal{A}}:=\left\{\{\bm{x}\in\mathbb{Z}^{m}\mid s_ix_i=0\}\,\middle|\,1\leq i\leq m-1\right\}\cup\left\{\{\bm{x}\in\mathbb{Z}^{m}\mid  x_i-x_j=0\}\,\middle|\,1\leq i<j\leq m-1\right\}.
\end{align*}
Depending on whether or not $x_m$ is included in the defining equations, the hyperplanes of $m$-arrangement $\mathcal{A}_m(\bm{s}_m)$ can be divided as
\begin{align*}
\mathcal{A}_m(\bm{s}_m)=\overline{\mathcal{A}}\cup\{\{s_mx_m=0\}\}\cup\{\{x_i-x_m=0\}\mid 1\leq i\leq m-1\}.
\end{align*}
Let $H=\{s_mx_m=0\}$.
Then $\mathscr{A}(\emptyset)=\mathcal{A}_m(\bm{s}_m)\setminus\{H\}$.
We define $H_q^a:=\{x_m=aq/d_m\}$ for $a\in\{0,1,\dots,d_m-1\}$.
Since $d_m=\gcd(s_m,q)$, we have the disjoint union $\displaystyle H_q=\bigsqcup_{a=0}^{d_m-1}H_q^a$.
The number of elements in the complements of $\mathcal{A}_m(\bm{s}_m)_q$ is equal to the number of elements in the complement of the affine arrangement $\overline{\mathcal{A}}_q\cup\{H^a_q\mid 0\leq a\leq d_m-1\}\cup\{\{x_i-x_m=0\}_q\mid 1\leq i\leq m-1\}$.
Hence this proof will focus on the latter one.
Let $a\in\{0,1,\dots,d_m-1\}$ and we consider the restriction on $H^a_q$.
For $1\leq i< m$, we have
\begin{align*}
\{x_i-x_m=0\}_q\cap H_q^a&=\{x_i=aq/d_m\}_q\cap H_q^a.
\end{align*}
In addition, since $s_m|s_i$, we have $\{x_i=aq/d_m\}_q\cap H_q^a\subseteq\{s_ix_i=0\}_q\cap H_q^a$.
In other words, for $\bm{x}\in H^a_q$, the condition $\bm{x}\not\in\{s_ix_i=0\}_q\cap H_q^a$ implies $\bm{x}\not\in\{x_i-x_m=0\}_q\cap H_q^a$ if $1\leq i<m$.
Therefore
\begin{align*}
\left|M_{(\mathcal{A}_m(\bm{s}_m)_q\setminus\{H_q^b\mid 0\leq b\leq a-1\})^{H^a_q}}(q)\right|=\left|H^a_q\setminus\bigcup_{H^{\prime}_q\in\overline{\mathcal{A}}_q}(H^{\prime}_q\cap H^a_q)\right|=\left|M_{\mathcal{A}_{m-1}(\bm{s}_{m-1})_q}(q)\right|,
\end{align*}
where $\mathcal{A}_{m-1}(\bm{s}_{m-1})$ is an $(m-1)$-arrangement in the equation above.
We recall that 
\begin{align*}
\left|M_{\mathcal{A}_{m}(\bm{s}_{m})_q}(q)\right|=\chi^k_{\mathcal{A}_{m}(\bm{s}_{m})}(q)=\prod_{i=1}^{m}(q-d_i-i+1)
\end{align*}
by Theorem \ref{thm-charquaipoly-defomA}.
We note that the value of $d_i$ changes depending on $k$.
By applying Corollary \ref{Cro-del-rest-gcd=1} for the hyperplanes $H^{0},H_1,\dots,H^{d_m-1}$ in this order, we have
\begin{align*}
\chi^k_{\mathscr{A}(\emptyset)}(q)&=
\left|M_{(\mathcal{A}_{m}(\bm{s}_{m})\setminus\{H\})_q}(q)\right|
=\left|M_{\mathcal{A}_{m}(\bm{s}_{m})_q}(q)\right|+\sum_{a=0}^{d_m-1}\left|M_{(\mathcal{A}_m(\bm{s}_m)_q\setminus\{H_q^b\mid 0\leq b\leq a-1\})^{H^a_q}}(q)\right|\\
&=\prod_{i=1}^{m}(q-d_i-i+1)+d_m\prod_{i=1}^{m-1}(q-d_i-i+1)
=(q-m+1)\prod_{i=1}^{m-1}(q-d_i-i+1).
\end{align*}

Next let $l>0$. We assume that the equality \eqref{eq-chiAmsm-deleting} holds for $|\mathscr{B}|=l-1$.
Let $H\in\mathscr{B}$ and we define $\mathscr{B}^{\prime}:=\mathscr{B}\setminus\{H\}$.
Since the arrangement $\mathscr{A}(\mathscr{B})$ is described as
\begin{align*}
\mathscr{A}(\mathscr{B})=\overline{\mathcal{A}}\cup\{\{x_i-x_m=0\}\mid 1\leq i\leq m-1\}\setminus\mathscr{B},
\end{align*}
we have $\displaystyle \mathscr{A}(\mathscr{B})_q^{H_q}=\left\{H^{\prime}_q\cap H_q\,\middle|\,H^{\prime}_q\in\overline{\mathcal{A}}_q\right\}$.
Then $\left|M_{\mathscr{A}(\mathscr{B})_q^{H_q}}(q)\right|=\left|H_q\setminus\bigcup_{H^{\prime}_q\in\overline{\mathcal{A}}_q}(H^{\prime}_q\cap H_q)\right|=\left|M_{\mathcal{A}_{m-1}(\bm{s}_{m-1})_q}(q)\right|$, where $\mathcal{A}_{m-1}(\bm{s}_{m-1})$ is an $(m-1)$-arrangement.
By the induction hypothesis,
\begin{align*}
\chi^k_{\mathscr{A}(\mathscr{B}^{\prime})}(q)=(q-m+l)\prod_{i=1}^{m-1}(q-d_i-i+1).
\end{align*}
Therefore since $\mathscr{A}(\mathscr{B})_q=\mathscr{A}(\mathscr{B}^{\prime})_q\setminus\{H_q\}$, we have
\begin{align*}
\chi_{\mathscr{A}(\mathscr{B})}^k(q)
=\left|M_{\mathscr{A}(\mathscr{B}^{\prime})_q}(q)\right|+\left|M_{\mathscr{A}(\mathscr{B})_q^{H_q}}(q)\right|
=(q-m+l+1)\prod_{i=1}^{m-1}(q-d_i-i+1)
\end{align*}
by Corollary \ref{Cro-del-rest-gcd=1}.
\end{proof}

\subsection{A sequence of subarrangements of deformation of type D}
Let $\bm{s}_m=(s_1,\dots,s_m)$ be a tuple of integers of length $m$, where $s_1,\dots,s_r$ are even, $s_{r+1}=\cdots=s_m=s$ is odd, and $s|s_r|\cdots|s_1$.
In this subsection, we give a sequence of subarrangements of $\mathcal{D}_m(\bm{s_m})$ that all $k$-constituents are factorized.
We recall that
\begin{align*}
\mathcal{D}_m(\bm{s_m})=\left\{\{s_ix_i=0\}\,\middle|\,1\leq i\leq m\right\}\cup\left\{\{x_i-x_j=0\},\{x_i+x_j=0\}\,\middle|\,1\leq i<j\leq m\right\}.
\end{align*}
Similarly to the proof of Theorem \ref{thm-deleting-typeA}, we define the subarrangement $\overline{\mathcal{D}}$ of the $m$-arrangement $\mathcal{D}_m(\bm{s}_m)$ by
\begin{align*}
\overline{\mathcal{D}}:=\left\{\{s_ix_i=0\}\,\middle|\,1\leq i\leq m-1\right\}\cup\left\{\{x_i-x_j=0\},\{x_i+x_j=0\}\,\middle|\,1\leq i<j\leq m-1\right\},
\end{align*}
where the hyperplanes in $\overline{\mathcal{D}}$ are subsets of $\mathbb{Z}_q^m$.
The arrangement $\overline{\mathcal{D}}$ will appear frequently in this subsection.
We first study the case when $r=m$.

\begin{Prop}\label{prop-chiC-deleting}
Let $s_1,\dots,s_m$ be even numbers with $s_m|s_{m-1}|\cdots|s_1$, and we set $\bm{s}_{m}=(s_1,\dots,s_m)$.
We take a subset $\mathscr{B}$ of $\{\{x_i-x_m=0\},\{x_i+x_m=0\}\mid 1\leq i\leq m-1\}$ with $|\mathscr{B}|=l$.
Let
\begin{align*}
\mathscr{A}(\mathscr{B})&=\mathcal{D}_m(s_m)\setminus\left(\left\{\left\{s_mx_m=0\right\}\right\}\cup\mathscr{B}\right).
\end{align*}
Then $\rho_{\mathscr{A}(\mathscr{B})}=s_1$ and the $k$-constituent of $\chi_{\mathscr{A}(\mathscr{B})}^{quasi}(q)$ is
\begin{align}\label{eq-chieven-deleting}
\chi^{k}_{\mathscr{A}(\mathscr{B})}(q)=(q-2m+l+2)\prod_{i=1}^{m-1}(q-d_i(k)-2i+2)
\end{align}
for $k|s_1$.
\end{Prop}
\begin{proof}
It is clear that $\rho_{\mathscr{A}(\mathscr{B})}=s_1$ since $s_1$ is even.
Let $q\in k+s_1\mathbb{Z}$ with $q\gg 0$.
We prove the equality \eqref{eq-chieven-deleting} by induction on $l$.
Let $\mathscr{B}=\emptyset$.
Let $H=\{s_mx_m=0\}$.
Then $\mathscr{A}(\emptyset)=\mathcal{D}_m(\bm{s}_m)\setminus\{\{s_mx_m=0\}\}$.
The hyperplanes of $\mathcal{D}_m(\bm{s_m})$ can be divided as
\begin{align*}
\mathcal{D}_m(\bm{s_m})=\overline{\mathcal{D}}\cup\{\{s_mx_m=0\}\}\cup\{\{x_i-x_m=0\},\{x_i+x_m=0\}\mid 1\leq i\leq m-1\}.
\end{align*}
We define $H_q^a:=\{x_m=aq/d_m\}$ for $a\in\{0,1,\dots,d_m-1\}$.
We recall that $d_m=\gcd(s_m,q)$ and $\displaystyle H_q=\bigsqcup_{a=0}^{d_m-1}H^a_q$.
Let $a\in\{0,1,\dots,d_m-1\}$.
For $1\leq i< m$, we have
\begin{align*}
\{x_i-x_m=0\}_q\cap H^a_q&=\{x_i=aq/d_m\}_q\cap H^a_q,\\
\{x_i+x_m=0\}_q\cap H^a_q&=\{x_i=-aq/d_m\}_q\cap H^a_q.
\end{align*}
Since $\{x_i=aq/d_m\}_q\cap H^a_q\subseteq\{s_ix_i=0\}_q\cap H^a_q$ and $\{x_i=-aq/d_m\}_q\cap H^a_q\subseteq\{s_ix_i=0\}_q\cap H^a_q$, we have
\begin{align*}
\left|M_{(\mathcal{D}_m(\bm{s}_m)_q\setminus\{H_q^b\mid 0\leq b\leq a-1\})^{H^a_q}}(q)\right|=\left|H^a_q\setminus\bigcup_{H^{\prime}_q\in\overline{\mathcal{D}}_q}(H^{\prime}_q\cap H^a_q)\right|=\left|M_{\mathcal{D}_{m-1}(\bm{s}_{m-1})_q}(q)\right|,
\end{align*}
where we note that $\mathcal{D}_{m-1}(\bm{s}_{m-1})$ is an $(m-1)$-arrangement.
Since $r=t=m$, we have
\begin{align*}
\chi^k_{\mathcal{D}_{m}(\bm{s}_{m})}(q)=\prod_{i=1}^{m}(q-d_i-2i+2)
\end{align*}
by Theorem \ref{thm-charquaipoly-defomD}.
We note that the value of $d_i$ changes depending on $k$.
Therefore
\begin{align*}
\chi^k_{\mathscr{A}(\emptyset)}(q)=
\chi^k_{\mathcal{D}_{m}(\bm{s}_{m})}(q)+d_m\chi^k_{\mathcal{D}_{m-1}(\bm{s}_{m-1})}(q)
=(q-2m+2)\prod_{i=1}^{m-1}(q-d_i-2i+2)
\end{align*}
by applying Corollary \ref{Cro-del-rest-gcd=1} for $H^0,H^1,\dots,H^{d_m-1}$.

Next let $l>0$, and we suppose that the equality \eqref{eq-chieven-deleting} holds for $|\mathscr{B}|=l-1$.
Let $H\in\mathscr{B}$ and we define $\mathscr{B}^{\prime}:=\mathscr{B}\setminus\{H\}$.
Here we can describe $\mathscr{A}(\mathscr{B})$ as
\begin{align*}
\mathscr{A}(\mathscr{B})=\overline{\mathcal{D}}\cup\{\{x_i-x_m=0\},\{x_i+x_m=0\}\mid 1\leq i\leq m-1\}\setminus\mathscr{B}.
\end{align*}
For $H^{\prime}\in\{\{x_i-x_m=0\},\{x_i+x_m=0\}\mid 1\leq i\leq m-1\}\setminus\mathscr{B}$, either $H^{\prime}_q\cap H_q$ is contained in $\overline{\mathcal{D}}_q^{H_q}$ or there exists an index $1\leq i\leq m-1$ such that $H^{\prime}_q\cap H_q=\{2x_i=0\}_q\cap H_q$.
Since $s_i$ is even, we have $\{2x_i=0\}_q\cap H_q\subseteq\{s_ix_i=0\}_q\cap H_q$.
Thus $\displaystyle \left|M_{\mathscr{A}(\mathscr{B})_q^{H_q}}(q)\right|
=\left|M_{\mathcal{D}_{m-1}(\bm{s}_{m-1})_q}(q)\right|$.
By the induction hypothesis and Corollary \ref{Cro-del-rest-gcd=1}, we have
\begin{align*}
\chi^k_{\mathscr{A}(\mathscr{B})}(q)
&=\left|M_{\mathcal{D}_{m}(\mathscr{A}(\mathscr{B}^{\prime})_q}(q)\right|+\left|M_{\mathscr{A}(\mathscr{B})^{H_q}_q}(q)\right|\\
&=(q-2m+l+1)\prod_{i=1}^{m-1}(q-d_i-2i+2)+\prod_{i=1}^{m-1}(q-d_i-2i+2)\\
&=(q-2m+l+2)\prod_{i=1}^{m-1}(q-d_i-2i+2).
\end{align*}
\end{proof}
Next we study the case when $0\leq r<m$.
Let $s_1,\dots,s_r$ be even and let $s$ be odd with $s|s_r|\cdots|s_1$.
We set a tuple $\bm{s}_m=(s_1,\dots,s_r,s,\dots,s)$ of length $m$, where its last $m-r$ entries are all $s$.
Before giving a sequence, we compute the characteristic quasi-polynomial for $\mathcal{D}_{m}(\bm{s}_{m})\cup\{\{2x_i=0\}\mid r+1\leq i\leq l\}$.
We use it to obtain the desired sequence.
The minimum period of $\chi_{\mathcal{D}_{m}(\bm{s}_{m})\cup\{\{2x_i=0\}\mid r+1\leq i\leq l\}}^{quasi}$ is $\lcm(s_1,2)=s_1$ which coincides with that of $\chi_{\mathcal{D}_{m}(\bm{s}_{m})}^{quasi}$.

\begin{Lem}\label{lem-char-quasi-poly-Ds+2x}
Let $0\leq r<m$.
Let $s_1,\dots,s_r$ be even numbers and let $s$ be an odd number with $s|s_r|\cdots|s_1$.
Let $\bm{s}_m=(s_1,\dots,s_r,s,\dots,s)$, where its last $m-r$ entries are all $s$.
Let $k|s_1$ and we write $d(k)=\gcd(s,k)$.
For any $r\leq l\leq m$, the $k$-constituent of $\chi_{\mathcal{D}_{m}(\bm{s}_{m})\cup\{\{2x_i=0\}\mid r+1\leq i\leq l\}}^{quasi}$ is
\begin{align*}
&\chi^k_{\mathcal{D}_{m}(\bm{s}_{m})\cup\{\{2x_i=0\}\mid r+1\leq i\leq l\}}(q)\\
=&
\begin{cases}
\prod_{i=1}^{r}(q-d_i(k)-2i+2)\prod_{i=r+1}^{m}(q-d(k)-2i+2)&\ \text{if $k$ is odd},\\
(q-d(k)-m-l+1)\prod_{i=1}^{r}(q-d_i(k)-2i+2)\prod_{i=r+1}^{m-1}(q-d(k)-2i+1)&\ \text{if $k$ is even}.
\end{cases}
\end{align*}
\end{Lem}
\begin{proof}
Let $q\in k+s_1\mathbb{Z}$ with $q\gg 0$.
If $k$ is odd, then $2$ is invertible in $\mathbb{Z}_q$.
So we have $\mathcal{D}_{m}(\bm{s}_{m})_q\cup\{\{2x_i=0\}_q\mid r+1\leq i\leq l\}\}=\mathcal{D}_{m}(\bm{s}_{m})_q$.
Therefore $\displaystyle \chi^k_{\mathcal{D}_{m}(\bm{s}_{m})\cup\{\{2x_i=0\}\mid r+1\leq i\leq l\}\}}(q)=\prod_{i=1}^{r}(q-d_i-2i+2)\prod_{i=r+1}^{m}(q-d-2i+2)$ by Theorem \ref{thm-charquaipoly-defomD}.

Next let $k$ be even.
Since we have the disjoint union $\{2x_i=0\}_q=\{x_i=0\}_q\sqcup\{x_i=q/2\}_q$, we have $\displaystyle \left|M_{\mathcal{D}_{m}(\bm{s}_{m})_q\cup \{\{2x_i=0\}_q\mid r+1\leq i\leq l\}\}}(q)\right|=\left|M_{\mathcal{D}_{m}(\bm{s}_{m})_q\cup\{\{x_i=q/2\}_q\mid r+1\leq i\leq l\}\}}(q)\right|$.
We prove the assertion by induction on $l$ and $m$.
If $l=r$, then $\mathcal{D}_{m}(\bm{s}_{m})_q\cup\{\{2x_i=0\}_q\mid r+1\leq i\leq l\}=\mathcal{D}_{m}(\bm{s}_{m})_q$.
So the assertion follows from Theorem \ref{thm-charquaipoly-defomD}.
If $m=r+1$, then $m-1\leq l\leq m$.
We note that, in this case, $\bm{s}_m=(s_1,\dots,s_{m-1},s)$ and $s_1,\dots,s_{m-1}$ are even.
If $l=m-1$, then $l=r$ while this case is already proved.
So let $l=m$ and we consider the arrangement $\mathcal{D}_{m}(\bm{s}_{m})_q\cup\{\{x_m=q/2\}_q\}$.
Let $H=\{x_m=q/2\}$.
Then we compute the restriction $\left(\mathcal{D}_{m}(\bm{s}_{m})_q\cup\{\{x_m=q/2\}_q\}\right)^{H_q}$.
Since
\begin{align*}
\mathcal{D}_{m}(\bm{s}_{m})=\overline{\mathcal{D}}\cup\{\{s_mx_m=0\}_q\}\cup\{\{x_i-x_m=0\},\{x_i+x_m=0\}\mid 1\leq i\leq m-1\},
\end{align*}
we have
\begin{align*}
\{x_i-x_m=0\}_q\cap H_q&=\{x_i=q/2\}_q\cap H_q\quad(i< m),\\
\{x_i+x_m=0\}_q\cap H_q&=\{x_i=q/2\}_q\cap H_q\quad(i< m).
\end{align*}
If $1\leq i\leq m-1$, then we have $\{x_i=q/2\}_q\cap H_q\subseteq \{s_ix_i=0\}\cap H_q$ since $s_i$ is even.
Therefore $\displaystyle \left|M_{(\mathcal{D}_{m}(\bm{s}_{m})_q\cup\{\{x_m=q/2\}_q\})^{H_q}}(q)\right|=\left|M_{\mathcal{D}_{m-1}(\bm{s}_{m-1})_q}(q)\right|$, where $\mathcal{D}_{m-1}(\bm{s}_{m-1})$ is an $(m-1)$-arrangement.
By Corollary \ref{Cro-del-rest-gcd=1}, we have
\begin{align*}
\left|M_{\mathcal{D}_{m}(\bm{s}_{m})_q\cup\{\{x_m=q/2\}_q\}}(q)\right|
=&\,\chi^k_{\mathcal{D}_{m}(\bm{s}_{m})}(q)-\chi^k_{\mathcal{D}_{m-1}(\bm{s}_{m-1})}(q)\\
=&\,(q-d-2m+2)\prod_{i=1}^{m-1}(q-d_i(k)-2i+2)-\prod_{i=1}^{m-1}(q-d_i(k)-2i+2)\\
=&\,(q-d-2m+1)\prod_{i=1}^{m-1}(q-d_i(k)-2i+2)
\end{align*}
while the assertion holds for $m=r+1$.

Let $l>r$ and $m>r+1$.
We assume that the assertion holds for $l-1$ and $m-1$.
Let $H=\{x_l=q/2\}$. Then
\begin{align*}
\{x_i-x_l=0\}_q\cap H_q&=\{x_i=q/2\}_q\cap H_q\quad(i\neq l),\\
\{x_i+x_l=0\}_q\cap H_q&=\{x_i=q/2\}_q\cap H_q\quad(i\neq l).
\end{align*}
We have $\{x_i=q/2\}_q\cap H_q\subseteq \{s_ix_i=0\}\cap H_q$ if $1\leq i\leq r$ and $\{x_i=q/2\}_q$ is not contained in $\mathcal{D}_{m}(\bm{s}_{m})_q\cup\{\{x_i=q/2\}_q\mid r+1\leq i\leq l\}$ if $l+1\leq i\leq m$.
So we have
\begin{align*}
\left|M_{(\mathcal{D}_{m}(\bm{s}_{m})_q\cup\{\{x_i=q/2\}_q\mid r+1\leq i\leq l\})^{H_q}}(q)\right|
=\left|M_{\mathcal{D}_{m-1}(\bm{s}_{m-1})_q\cup\{\{x_i=q/2\}_q\mid r+1\leq i\leq m-1\}}(q)\right|,
\end{align*}
where $\mathcal{D}_{m-1}(\bm{s}_{m-1})_q\cup\{\{x_i=q/2\}_q\mid r+1\leq i\leq m-1\}$, which arrears in the right hand side, is an $(m-1)$-arrangement.
By the induction hypothesis,
\begin{align*}
&\chi^k_{\mathcal{D}_{m}(\bm{s}_{m})\cup\{\{2x_i=0\}\mid r+1\leq i\leq l-1\}}(q)=(q-d-m-l+2)\prod_{i=1}^{r}(q-d_i-2i+2)\prod_{i=r+1}^{m-1}(q-d-2i+1),\\
&\chi^k_{\mathcal{D}_{m-1}(\bm{s}_{m-1})\cup\{\{2x_i=0\}\mid r+1\leq i\leq m-1\}}(q)=\prod_{i=1}^{r}(q-d_i-2i+2)\prod_{i=r+1}^{m-1}(q-d-2i+1).
\end{align*}
Therefore we have
\begin{align*}
&\chi^k_{\mathcal{D}_{m}(\bm{s}_{m})\cup \{2x_i=0\}\mid r+1\leq i\leq l\}}(q)\\
=&\,\chi^k_{\mathcal{D}_{m}(\bm{s}_{m})\cup\{\{2x_i=0\}\mid r+1\leq i\leq l-1\}}(q)-\chi^k_{\mathcal{D}_{m-1}(\bm{s}_{m-1})\cup\{\{2x_i=0\}\mid r+1\leq i\leq m-1\}}(q)\\
=&\,(q-d-m-l-1)\prod_{i=1}^{r}(q-d_i-2i+2)\prod_{i=r+1}^{m-1}(q-d-2i+1)
\end{align*}
by Corollary \ref{Cro-del-rest-gcd=1}.
\end{proof}

\begin{Them}\label{thm-chieo-deleting}
Let $0\leq r<m$.
Let $s_1,\dots,s_r$ be even numbers and let $s$ be an odd number with $s|s_r|\cdots|s_1$.
Let $\bm{s}_m=(s_1,\dots,s_r,s,\dots,s)$, where its last $m-r$ entries are all $s$.
Let
\begin{align*}
\mathscr{A}_1(l)&=\mathcal{D}_m(\bm{s}_m)\setminus\left\{\left\{x_i-x_m=0\right\}\,\middle|\,r+1\leq i\leq l\right\}\quad(r\leq l\leq m-1),\\
\mathscr{A}_2(l)&=\mathscr{A}_1(m-1)\setminus\left(\left\{\left\{sx_m=0\right\}\right\}\cup\left\{\left\{x_i+x_m=0\right\}\,\middle|\,r+1\leq i\leq l\right\}\right)\quad(r\leq l\leq m-1),\\
\mathscr{A}_3(l)&=\mathscr{A}_2(m-1)\setminus\left\{\left\{x_i-x_m=0\right\}\,\middle|\,1\leq i\leq l\right\}\quad(0\leq l\leq r),\\
\mathscr{A}_4(l)&=\mathscr{A}_3(r)\setminus\left\{\left\{x_i+x_m=0\right\}\,\middle|\,1\leq i\leq l\right\}\quad(0\leq l\leq r).
\end{align*}
For $k|s_1$, we write $d(k)=\gcd(s,k)$. We define
\begin{align*}
P_k(q):=
\begin{cases}
\prod_{i=1}^{r}(q-d_i(k)-2i+2)\prod_{i=r+1}^{m-1}(q-d(k)-2i+2)&\ \text{if $k$ is odd},\\
\prod_{i=1}^{r}(q-d(k)-2i+2)\prod_{i=r+1}^{m-2}(q-d(k)-2i+1)&\ \text{if $k$ is even}.
\end{cases}
\end{align*}
Then $\rho_{\mathscr{A}_1(l)}=\rho_{\mathscr{A}_2(l)}=\rho_{\mathscr{A}_3(l)}=\rho_{\mathscr{A}_4(l)}=s_1$ and the $k$-constituents are described as follows.
\begin{align*}
(1)\ \chi^k_{\mathscr{A}_1(l)}(q)&=
\begin{cases}
(q-d(k)-2m-r+l+2)P_k(q)&\ \text{if $k$ is odd},\\
(q-d(k)-2m-r+l+3)(q-d(k)-m-r+1)P_k(q)&\ \text{if $k$ is even}.
\end{cases}
\\
(2)\ \chi^k_{\mathscr{A}_2(l)}(q)&=
\begin{cases}
(q-m-2r+l+1)P_k(q)&\ \text{if $k$ is odd},\\
(q-m-2r+l+1)(q-d(k)-m-r+2)P_k(q)&\ \text{if $k$ is even}.
\end{cases}
\\
(3)\ \chi^k_{\mathscr{A}_3(l)}(q)&=
\begin{cases}
(q-2r+l)P_k(q)&\ \text{if $k$ is odd},\\
(q-2r+l)(q-d(k)-m-r+2)P_k(q)&\ \text{if $k$ is even}.
\end{cases}
\\
(4)\ \chi^k_{\mathscr{A}_4(l)}(q)&=
\begin{cases}
(q-r+l)P_k(q)&\ \text{if $k$ is odd},\\
(q-r+l)(q-d(k)-m-r+2)P_k(q)&\ \text{if $k$ is even}.
\end{cases}
\end{align*}
\end{Them}
\begin{proof}
$(1)$ We prove the assertion by induction on $l$.
Let $q\in k+s_1\mathbb{Z}$ with $q\gg 0$.
We first prove the assertion for $l=r$.
Since $\mathscr{A}_1(r)=\mathcal{D}_m(\bm{s}_m)$, by Theorem \ref{thm-charquaipoly-defomD}, the $k$-constituent of $\chi_{\mathscr{A}_1(r)}^{quasi}$ is
\begin{align*}
\chi^k_{\mathscr{A}_1(r)}(q)&=
\begin{cases}
(q-d-2m-2)P_k(q)&\ \text{if $k$ is odd},\\
(q-d-2m+3)(q-d-m-r+1)P_k(q)&\ \text{if $k$ is even}
\end{cases}
\end{align*}
while the assertion holds.

Let $l>r$, and we assume that the assertion holds for $l-1$.
Let $H=\{x_l-x_m=0\}$.
Then $\mathscr{A}_1(l)=\mathscr{A}_1(l-1)\setminus\{H\}$.
The arrangement $\mathscr{A}_1(l-1)$ is described as
\begin{align*}
\mathscr{A}_1(l-1)
=&\,\overline{\mathcal{D}}\cup\{\{sx_m=0\}\}
\cup\left\{\{x_i-x_m=0\}\,\middle|\,1\leq i\leq r,l\leq i\leq m-1\right\}\\
&\,\cup\left\{\{x_i+x_m=0\}\,\middle|\,1\leq i\leq m-1\right\}.
\end{align*}
As elements in $\mathscr{A}_1(l-1)_q^{H_q}$, we have the following equalities:
\begin{align*}
\{sx_m=0\}_q\cap H_q&=\{sx_l=0\}_q\cap H_q,\\
\{x_i-x_m=0\}_q\cap H_q&=\{x_i-x_l=0\}_q\cap H_q\quad(1\leq i\leq r,l<i\leq m-1),\\
\{x_i+x_m=0\}_q\cap H_q&=
\begin{cases}
\{x_i+x_l=0\}_q\cap H_q&\quad(i\neq l),\\
\{2x_l=0\}_q\cap H_q&\quad(i=l).
\end{cases}
\end{align*}
These elements belong to $\overline{\mathcal{D}}_q^{H_q}$ except for $\{2x_l=0\}_q\cap H_q$.
Then we have $\displaystyle \left|M_{\mathscr{A}(l-1)_q^{H_q}}(q)\right|=\left|M_{\mathcal{D}_{m-1}(\bm{s}_{m-1})_q\cup\{\{2x_l=0\}_q\}}(q)\right|$ while this coincides with $\displaystyle \left|M_{\mathcal{D}_{m-1}(\bm{s}_{m-1})_q\cup\{\{2x_{r+1}=0\}_q\}}(q)\right|$.
By Lemma \ref{lem-char-quasi-poly-Ds+2x}, we have
\begin{align*}
\left|M_{\mathscr{A}_1(l-1)_q^{H_q}}(q)\right|=\chi^k_{\mathcal{D}_{m-1}(\bm{s}_{m-1})\cup\{\{2x_{r+1}=0\}\}}(q)=
\begin{cases}
P_k(q)&\ \text{if $k$ is odd},\\
(q-d-m-r+1)P_k(q)&\ \text{if $k$ is even}.
\end{cases}
\end{align*}
By the induction hypothesis,
\begin{align*}
\left|M_{\mathscr{A}_1(l-1)_q}(q)\right|=\chi^k_{\mathscr{A}_1(l-1)}(q)=
\begin{cases}
(q-d-2m-r+l+1)P_k(q)&\ \text{if $k$ is odd},\\
(q-d-2m-r+l+2)(q-d-m-r+1)P_k(q)&\ \text{if $k$ is even}.
\end{cases}
\end{align*}
Therefore we have
\begin{align*}
\chi^k_{\mathscr{A}_1(l)}(q)&=
\chi^k_{\mathscr{A}_1(l-1)}(q)+\chi^k_{\mathcal{D}_{m-1}(\bm{s}_{m-1})\cup\{\{2x_{r+1}=0\}\}}(q)\\
&=
\begin{cases}
(q-d-2m-r+l+1)P_k(q)&\ \text{if $k$ is odd},\\
(q-d-2m-r+l+3)(q-d-m-r+1)P_k(q)&\ \text{if $k$ is even}.
\end{cases}
\end{align*}
by Corollary \ref{Cro-del-rest-gcd=1}.

$(2)$ We also prove the assertion by induction on $l$.
Let $l=r$. Let $q\in k+s_1\mathbb{Z}$.
We set $H=\{sx_m=0\}$.
Then we have $\mathscr{A}_2(r)=\mathscr{A}_1(m-1)\setminus \{H\}$.
Let $H_q^a:=\{x_m=aq/d_m\}_q$ for $a\in\{0,1,\dots,d_m-1\}$.
We recall that 
\begin{align*}
\mathscr{A}_1(m-1)=\overline{\mathcal{D}}\cup\{H\}
\cup\left\{\{x_i-x_m=0\}\,\middle|\,1\leq i\leq r\right\}
\cup\left\{\{x_i+x_m=0\}\,\middle|\,1\leq i\leq m-1\right\}
\end{align*}
and $\displaystyle H_q=\bigsqcup_{a=0}^{d-1}H_q^a$.
For any $a\in\{0,1,\dots,d-1\}$ and $1\leq i< m$, the elements
\begin{align*}
\{x_i-x_m=0\}_q\cap H_q^a&=\{x_i=aq/d\}_q\cap H_q^a,\\
\{x_i+x_m=0\}_q\cap H_q^a&=\{x_i=-aq/d\}_q\cap H_q^a
\end{align*}
are contained in some hyperplanes in $\overline{\mathcal{D}}_q^{H_q^a}$.
So we have $\displaystyle \left|M_{\mathscr{A}_1(m-1)_q^{H_q^a}}(q)\right|=\left|M_{\mathcal{D}_{m-1}(\bm{s}_{m-1})_q}(q)\right|$.
Here
\begin{align*}
\chi^k_{\mathcal{D}_{m-1}(\bm{s}_{m-1})}(q)=
\begin{cases}
P_k(q)&\ \text{if $k$ is odd},\\
(q-d-m-r+2)P_k(q)&\ \text{if $k$ is even}
\end{cases}
\end{align*}
and
\begin{align*}
\chi^k_{\mathscr{A}_1(m-1)}(q)=
\begin{cases}
(q-d-m-r+1)P_k(q)&\ \text{if $k$ is odd},\\
(q-d-m-r+2)(q-d-m-r+1)P_k(q)&\ \text{if $k$ is even}.
\end{cases}
\end{align*}
Therefore, by Corollary \ref{Cro-del-rest-gcd=1},
\begin{align*}
\chi^k_{\mathscr{A}_2(r)}(q)&=
\left|M_{\mathscr{A}_1(m-1)_q}(q)\right|+\sum_{a=0}^{d-1}\left|M_{\mathscr{A}_1(m-1)_q^{H_q^a}}(q)\right|
=\chi^k_{\mathscr{A}_1(m-1)}(q)+d\chi^k_{\mathcal{D}_{m-1}(\bm{s}_{m-1})}(q)\\
&=
\begin{cases}
(q-m-r+1)P_k(q)&\ \text{if $k$ is odd},\\
(q-m-r+1)(q-d-m-r+2)P_k(q)&\ \text{if $k$ is even}.
\end{cases}
\end{align*}

Next let $l>r$, and we assume that the assertion holds for $l-1$.
We set $H=\{x_l+x_m=0\}$.
Then $\mathscr{A}_2(l)=\mathscr{A}_2(l-1)\setminus\{H\}$.
We recall that
\begin{align*}
\mathscr{A}_2(l-1)=\overline{\mathcal{D}}\cup\left\{\{x_i-x_m=0\}\,\middle|\,1\leq i\leq r\right\}\cap\{\{x_i+x_m=0\}\mid 1\leq i\leq r,l\leq i\leq m-1\}
\end{align*}
and then $\displaystyle \left|M_{\mathscr{A}_2(l-1)_q^{H_q}}(q)\right|=\left|M_{\mathcal{D}_{m-1}(\bm{s}_{m-1})_q}(q)\right|$.
By the induction hypothesis,
\begin{align*}
\chi^k_{\mathscr{A}_2(l-1)}(q)=
\begin{cases}
(q-m-2r+l)P_k(q)&\ \text{if $k$ is odd},\\
(q-m-2r+l)(q-d-m-r+2)P_k(q)&\ \text{if $k$ is even}.
\end{cases}
\end{align*}
Therefore we have
\begin{align*}
\chi^k_{\mathscr{A}_2(l)}(q)&=
\chi^k_{\mathscr{A}_2(l-1)}(q)+\chi^k_{\mathcal{D}_{m-1}(\bm{s}_{m-1})}(q)\\
&=
\begin{cases}
(q-m-2r+l+1)P_k(q)&\ \text{if $k$ is odd},\\
(q-m-2r+l+1)(q-d-m-r+2)P_k(q)&\ \text{if $k$ is even}
\end{cases}
\end{align*}
by Corollary \ref{Cro-del-rest-gcd=1}.

$(3)$ Let $H=\{x_l-x_m=0\}$.
Then we obtain $\displaystyle \left|M_{\mathscr{A}_3(l-1)_q^{H_q}}(q)\right|=\left|M_{\mathcal{D}_{m-1}(\bm{s}_{m-1})_q\cup\{\{2x_l=0\}_q\}}(q)\right|$ from the same arguments as in $(1)$ and $(2)$.
Since $s_l$ is even, $\{2x_l=0\}_q\cap H_q\subseteq\{s_lx_l=0\}_q\cap H_q$ while we have $\displaystyle \left|M_{\mathcal{D}_{m-1}(\bm{s}_{m-1})_q\cup\{\{2x_l=0\}_q\}}(q)\right|=\left|M_{\mathcal{D}_{m-1}(\bm{s}_{m-1})_q}(q)\right|$.
Therefore the assertion follows from the induction on $l$.

$(4)$ Let $H=\{x_l+x_m=0\}$.
Since $\displaystyle \left|M_{\mathscr{A}_3(l-1)_q^{H_q}}(q)\right|=\left|M_{\mathcal{D}_{m-1}(\bm{s}_{m-1})_q}(q)\right|$, the proof is similar to $(1)$, $(2)$, and $(3)$. 
\end{proof}
Since $\mathscr{A}_4(r)=\overline{\mathcal{D}}$, Theorem \ref{thm-chieo-deleting} gives a sequence of arrangements such that all $k$-constituents are factorized, one by one from $\mathcal{D}_{m}(\bm{s}_{m})$, deleting hyperplanes to $\overline{\mathcal{D}}$ which can be regarded as $\mathcal{D}_{m-1}(\bm{s}_{m-1})$.

\begin{Examp}
Let $m=3$, $r=2$, $\bm{s}_3=(12,3,3)$.
We list $k$-constituents obtained in Theorem \ref{thm-chieo-deleting}.
The hyperplanes to be deleted are
\begin{align*}
\{x_2-x_3=0\},\ \{3x_3=0\},\ \{x_2+x_3=0\},\ \{x_1-x_3=0\},\ \{x_1+x_3=0\},
\end{align*}
in that order:
\begin{align*}
&\chi^{6}_{\mathcal{D}_{3}(\bm{s}_{3})}(q) = \left(q - 6\right)^{3},\\
&\chi^{6}_{\mathscr{A}_1(3)}(q) =\left(q - 6\right)^{2} \left(q - 5\right),\\
&\chi^{6}_{\mathscr{A}_2(2)}(q) =\left(q - 6\right) \left(q - 5\right) \left(q - 3\right),\\
&\chi^{6}_{\mathscr{A}_2(3)}(q) =\left(q - 6\right) \left(q - 5\right) \left(q - 2\right),\\
&\chi^{6}_{\mathscr{A}_3(1)}(q) =\left(q - 6\right) \left(q - 5\right) \left(q - 1\right),\\
&\chi^{6}_{\mathscr{A}_3(2)}(q) =q \left(q - 6\right) \left(q - 5\right).
\end{align*}
\end{Examp}
Changing the order of the hyperplanes to be deleted may result in the characteristic quasi-polynomial not being factorized. In particular, if $\{sx_m=0\}$ is deleted first, then $k$-constituents will not factorize.
\begin{Examp}
Let $m=4$, $r=2$, and $\bm{s}_4=(4,2,1,1)$.
Then the $2$-constituent of the characteristic quasi-polynomial for the arrangement obtained by deleting $\{x_4=0\}$ from $\mathcal{D}_{4}(\bm{s}_{4})$ is
\begin{align*}
\chi^{2}_{(\mathcal{D}_{4}(\bm{s}_{4})\setminus\{x_4=0\})}(q) = \left(q - 4\right) \left(q - 2\right) \left(q^{2} - 11 q + 31\right).
\end{align*}
\end{Examp}

\subsection{Inductive poset}
Let $(\mathcal{P},\leq)$ be a finite poset with a unique minimal element $\hat{0}$.
A poset $\mathcal{P}$ is ranked if it satisfies that for any $x\in\mathcal{P}$, all maximal chains among those with $x$ as greatest element have the same finite length, while we denote the length by $\rank(x)$.
We define the rank of a poset $\mathcal{P}$ to be $\rank(\mathcal{P}):=\max \{\rank(x) \mid x \in \mathcal{P}\}$.
An element $a\in\mathcal{P}$ with $\rank(a)=1$ is called an atom.
We define by $\Atom(\mathcal{P})$ the set of atoms of $\mathcal{P}$.
For $x,y\in \mathcal{P}$, we define the join $x \vee y$ and the meet $x \wedge y$ as follows:
\begin{align*}
x \vee y:=\min \{z\in \mathcal{P} \mid z\ge x\ \text{and}\ z \ge y \},\ 
x \wedge y:=\max \{z \in \mathcal{P} \mid z \le x\ \text{and}\ z \le y \}.
\end{align*}
The poset $\mathcal{P}$ is said to be a lattice if $|x \vee y| = 1$ and $|x \wedge y| = 1$ for any $x, y \in \mathcal{P}$.
We write $x<\mathrel{\mkern-5mu}\mathrel{\cdot} y$ if $x<y$ and $x\le z<y$ implies $x = z$.
A lattice $\mathcal{P}$ is called geometric if it satisfies the following condition: $x<\mathrel{\mkern-5mu}\mathrel{\cdot} y$ if and only if there exists an atom $a \in \mathcal{P}$ such that $a \not\le x$ and $\{y\} = x \vee a$.
A poset $\mathcal{P}$ is said to be locally geometric if $\mathcal{P}_{\leq x}:=\{y\in\mathcal{P}\mid y\leq x\}$ is a geometric lattice for any $x\in\mathcal{P}$.
If $\mathcal{P}$ is locally geometric, then $\mathcal{P}_{\leq x}$ and $\mathcal{P}_{\geq x}:=\{y\in\mathcal{P}\mid y\geq x\}$ are geometric lattices for all $x\in\mathcal{P}$ \cite[Remark 2.2.6]{Bibby-Delucchi2022}.
We define the M\"{o}bius function $\mu$ and the characteristic polynomial $\chi_{\mathcal{P}}(q)$ of the poset$\mathcal{P}$ as follows:
\begin{align*}
\mu(a,b)&:=
\begin{cases}
0&\ \text{if}\ a\not\leq b,\\
1&\ \text{if}\ a=b,\\
-\sum_{a\leq c< b}\mu(a,c)&\ \text{if}\ a<b
\end{cases}
\qquad \text{for}\ a,b\in \mathcal{P},\\
\chi_{\mathcal{P}}(q)&:=\sum_{x\in\mathcal{P}}\mu(\hat{0},x)q^{\rank(\mathcal{P})-\rank(x)}.
\end{align*}
Let $\mathcal{P}$ be a locally geometric poset and let $a\in\Atom(\mathcal{P})$. 
Then $\mathcal{P}^{\prime}$ is defined by the poset consisting of the minimal element $\hat{0}$ and all possible joins of the elements in $\Atom(\mathcal{P})\setminus\{a\}$.
In addition, $\mathcal{P}^{\prime\prime}:=\mathcal{P}_{\geq a}$.
The following class is defined by Pagaria, Pismataro, Tran, and Vecchi \cite{PPTV2024}.
\begin{Defi}
The class $\mathbf{IP}$ of inductive posets is the smallest class of locally geometric posets which satisfies
\begin{itemize}
\item[(1)]\ $\{\hat{0}\}\in\mathbf{IP}$, 
\item[(2)]\ if there exists an atom $a\in\Atom(\mathcal{P})$ such that $P^{\prime}\in\mathbf{IP}$, $P^{\prime\prime}\in\mathbf{IP}$, and $\chi_{P^{\prime\prime}}(q)$ divides $\chi_{P^{\prime}}(q)$, then $P\in\mathbf{IP}$.
\end{itemize}
\end{Defi}
We now explain how to regard a $k$-constituent as a characteristic polynomial of some poset (see \cite{Liu-Tran-Yoshinaga21,PPTV2024,Tran-yoshinaga2019}).
For a given tuple $\alpha=(a_1,\dots,a_m)\in\mathbb{Z}^m$, we can define the hypertorus $H_{\alpha,\mathbb{S}^1}$ by
\begin{align*}
H_{\alpha,\mathbb{S}^1}:=\{\bm{t}\in(\mathbb{S}^1)^m\mid t_1^{a_1}\cdots t_m^{a_m}=1\}.
\end{align*}
For a set $A \subseteq \mathbb{Z}^m$, we define the central toric arrangement $\mathscr{A}_{\mathbb{S}^1}$ by
\begin{align*}
\mathscr{A}_{\mathbb{S}^1}:=\{\text{connected components of}\ H_{\alpha,\mathbb{S}^1}\mid \alpha\in A\}.
\end{align*}
Let $L$ be the set of connected components of nonempty intersections of elements in $\mathscr{A}_{\mathbb{S}^1}$.
Then $L$ is a poset by the reverse inclusion and $L$ is ranked by the dimension of its elements (layers).
It is known that $L$ is a locally geometric poset (\cite[Corollary 4.4.6]{Bibby-Delucchi2022}).
For any $k\in \mathbb{Z}$, the map $k$ is defined by $k:(\mathbb{S}^1)\rightarrow(\mathbb{S}^1)$ via $t\mapsto t^k$.
Define $L[k]:=\{x \in L\mid \bm{1}\in k(x)\}$.
\begin{Them}[Tran and Yoshinaga \cite{Tran-yoshinaga2019}]
For any $k|\rho_{\mathscr{A}}$, the $k$-constituent of $\chi_{\mathscr{A}}^{quasi}(q)$ coincides with the characteristic polynomial of $L[k]$, that is, $\chi_{\mathscr{A}}^k(q)=\chi_{L[k]}(q)$.
\end{Them}
From the above, it follows that the poset obtained by torusization of the arrangement treated in this paper is an inductive poset.
We note that the hyperplane $\{s_ix_i=0\}$ corresponds to the hypertori $\{t_i=\xi_i^j\}$ $(j=0,1,\dots,d_i-1)$, where $\xi_i$ is $d_i$th roots of unity.
\begin{Cor}
\ 
\begin{itemize}
\item[$(1)$]\ Let $\bm{s}_m=(s_1,\dots,s_m)$ be a tuple of integers with $s_m|\cdots|s_1$.
Let $L$ be the set of connected components of nonempty intersections of elements in the toric arrangement corresponding to $\mathcal{A}_m(\bm{s}_m)$.
Then $L[k]\in\mathbf{IP}$ for any $k|s_1$.
\item[$(2)$]\ Let $s_1,\dots,s_r$ be even numbers, let $s$ be an odd number with $s|s_r|\cdots|s_1$, and let $\bm{s}_m=(s_1,\dots,s_r,s,\dots,s)$ be a tuple of length $m$.
Let $L$ be the set of connected components of nonempty intersections of elements in the toric arrangement corresponding to $\mathcal{D}_m(\bm{s}_m)$.
Then $L[k]\in\mathbf{IP}$ for any $k|s_1$.
\end{itemize}
\end{Cor}
\begin{proof}
The proofs is the same as in Theorem \ref{thm-deleting-typeA}, Proposition \ref{prop-chiC-deleting}, and Theorem \ref{thm-chieo-deleting}.
\end{proof}

\section{Other examples}\label{sec-other-examples}
In this section, we study characteristic quasi-polynomials for hyperplane arrangements obtained by deleting several hyperplanes from the Coxeter arrangement, but are not factorized into linear polynomials.
\begin{Prop}\label{prop-chiB-deleting2}
Let $m\geq 2$. Let
\begin{align*}
\mathscr{A}(l)&=\mathcal{B}_m\setminus\left(\left\{\left\{x_m=0\right\}\right\}\cup\left\{\left\{x_i-x_m=0\right\}\,\middle|\,1\leq i\leq l\right\}\right)\quad(0\leq l\leq m-1),\\
\mathscr{A}^{\prime}(l)&=\mathscr{A}(m-1)\setminus\left\{\left\{x_j+x_m=0\right\}\,\middle|\,1\leq j\leq l\right\}\quad(0\leq l\leq m-1).
\end{align*}
Then the characteristic quasi-polynomials are described as follows.
\begin{itemize}
\item[$(1)$]\ $\displaystyle \chi_{\mathscr{A}(l)}^{quasi}(q)=
\begin{cases}
(q-2m+l+2)\prod_{i=1}^{m-1}(q-2i+1)&\quad\text{if $q$ is odd},\\
\left(q^2-(3m-l-3)q+2m^2-(l+3)m+1\right)\prod_{i=1}^{m-2}(q-2i)&\quad\text{if $q$ is even}.
\end{cases}$
\item[$(2)$]\ $\displaystyle \chi_{\mathscr{A}^{\prime}(l)}^{quasi}(q)=
\begin{cases}
(q-m+l+1)\prod_{i=1}^{m-1}(q-2i+1)&\quad\text{if $q$ is odd},\\
(q^2-(2m-l-2)q+m^2-(l+2)m+l+1)\prod_{i=1}^{m-2}(q-2i)&\quad\text{if $q$ is even}.
\end{cases}$
\end{itemize}
\end{Prop}
\begin{proof}
$(1)$\ We prove the assertion by induction on $l$.
Let $l=0$ and let $H=\{x_m=0\}$.
Then $\mathscr{A}(0)_q=(\mathcal{B}_m)_q\setminus\{H_q\}$ and $\left(\mathcal{B}_m\right)_q^{H_q}=(\mathcal{B}_{m-1})_q$.
We recall that
\begin{align*}
\chi_{\mathcal{B}_m}^{quasi}(q)=
\begin{cases}
\prod_{i=1}^{m}(q-2i+1)&\quad\text{if $q$ is odd},\\
(q-m)\prod_{i=1}^{m-1}(q-2i)&\quad\text{if $q$ is even}.
\end{cases}
\end{align*}
Therefore, by Corollary \ref{Cro-del-rest-gcd=1}, we have
\begin{align*}
\chi_{\mathscr{A}(0)}^{quasi}(q)&=
\chi_{\mathcal{B}_m}^{quasi}(q)+\chi_{\mathcal{B}_{m-1}}^{quasi}(q)\\
&=
\begin{cases}
(q-2m+2)\prod_{i=1}^{m-1}(q-2i+1)&\quad\text{if $q$ is odd},\\
(q^2-(3m-3)q+2m^2-3m+1)\prod_{i=1}^{m-2}(q-2i)&\quad\text{if $q$ is even}
\end{cases}
\end{align*}
since $(q-m)(q-2m+2)+(q-m+1)=q^2-(3m-3)q+2m^2-3m+1$ when $q$ is even.

Next let $l>0$.
Let $H=\{x_l-x_m=0\}$.
Then $\mathscr{A}(l)_q=\mathscr{A}(l-1)_q\setminus\{H_q\}$ and $\mathscr{A}(l-1)_q^{H_q}=(\mathcal{B}_{m-1})_q\cup\{\{2x_l=0\}_q\}$.
By Lemma \ref{lem-char-quasi-poly-Ds+2x} of the case when $r=0$, $\bm{s}_{m-1}=(1,\dots,1)$, and $l=1$, we have
\begin{align*}
\chi_{\mathcal{B}_{m-1}\cup\{\{2x_l=0\}\}}^{quasi}(q)=
\begin{cases}
\prod_{i=1}^{m-1}(q-2i+1)&\quad\text{if $q$ is odd},\\
(q-m)\prod_{i=1}^{m-2}(q-2i)&\quad\text{if $q$ is even}
\end{cases}
\end{align*}
since $d(k)=1$ for $k=1,2$ in this case.
By the induction hypothesis and Corollary \ref{Cro-del-rest-gcd=1}, we have
\begin{align*}
\chi_{\mathscr{A}(l)}^{quasi}(q)&=
\chi_{\mathscr{A}(l-1)}^{quasi}(q)+\chi_{\mathcal{B}_{m-1}\cup\{\{2x_l=0\}\}}^{quasi}(q)\\
&=
\begin{cases}
(q-2m+l+2)\prod_{i=1}^{m-1}(q-2i+1)&\quad\text{if $q$ is odd},\\
\left(q^2-(3m-l-3)q+2m^2-(l+3)m+1\right)\prod_{i=1}^{m-2}(q-2i)&\quad\text{if $q$ is even}.
\end{cases}
\end{align*}

$(2)$ For $1\leq l\leq m-1$, let $H=\{x_l-x_m=0\}$.
Then $\mathscr{A}^{\prime}(l)_q=\mathscr{A}^{\prime}(l-1)_q\setminus\{H_q\}$ and $\mathscr{A}^{\prime}(l-1)_q^{H_q}=(\mathcal{B}_{m-1})_q$.
The assertion holds by induction on $l$.
\end{proof}

\begin{Lem}\label{lem-case-21...1and2}
Let $m\geq 3$ and let $\mathscr{A}=\{\{2x_1=0\}\}\cup\mathcal{D}_m$. Then
\begin{align*}
\chi_{\mathscr{A}}^{quasi}(q)=
\begin{cases}
(q-m)\prod_{i=1}^{m-1}(q-2i+1)&\quad\text{if $q$ is odd},\\
\left(q^2-2mq+(m-1)(m+2)\right)\prod_{i=1}^{m-2}(q-2i)&\quad\text{if $q$ is even}.
\end{cases}
\end{align*}
\end{Lem}
\begin{proof}
We define a tuple $\bm{s}$ of length $1$ by $\bm{s}=(2)$.
Then $\mathscr{A}=\mathcal{D}_m(\bm{s})$ with $r=t=1$.
We have $\rho_{\mathcal{D}_{m}(\bm{s})}=\lcm(s_1,2)=2$ and $\displaystyle d_{1}=\gcd(k,s_1)=
\begin{cases}
1\ &\text{if}\ k=1,\\
2\ &\text{if}\ k=2.
\end{cases}$
By Theorem \ref{thm-charquaipoly-defomD}, if $q$ is odd, then
\begin{align*}
\chi_{\mathcal{D}_{m}(\bm{s})}^{1}(q)
&=(q-d_1)\left(\prod_{i=2}^m(q-2i+1)+(m-1)\prod_{i=2}^{m-1}(q-2i+1)\right)\\
&=(q-1)(q-m)\prod_{i=2}^{m-1}(q-2i+1)
=(q-m)\prod_{i=2}^{m-1}(q-2i+1).
\end{align*}
If $q$ is even, then
\begin{align*}
\chi_{\mathcal{D}_{m}(\bm{s})}^{2}(q)
&=(q-d_1)\left(\prod_{i=2}^m(q-2i)+2(m-1)\prod_{i=2}^{m-1}(q-2i)+(m-1)(m-2)\prod_{i=2}^{m-2}(q-2i)\right)\\
&=(q-2)\left((q-2m)(q-2m+2)+2(m-1)(q-2m+2)+(m-1)(m-2)\right)\prod_{i=2}^{m-2}(q-2i)\\
&=(q^2-2mq+(m-1)(m+2))\prod_{i=1}^{m-2}(q-2i).
\end{align*}
\end{proof}

\begin{Prop}\label{prop-chiD-deleting}
Let $m\geq 4$ and let
\begin{align*}
\mathscr{A}(l)&=\mathcal{D}_m\setminus\left\{\left\{x_i-x_m=0\right\}\,\middle|\,1\leq i\leq l\right\}\quad(0\leq l\leq m-1),\\
\mathscr{A}^{\prime}(l)&=\mathscr{A}(m-1)\setminus\left\{\left\{x_j+x_m=0\right\}\,\middle|\,1\leq j\leq l\right\}\quad(0\leq l\leq m-1).
\end{align*}
Let
\begin{align*}
f_l(q):=q^3-(4m-6-l)q^2+(m-1)(5m-8-2l)q-(m-2)(2m^2-(l+2)m-l).
\end{align*}
Then the characteristic quasi-polynomials are described as follows.
\begin{itemize}
\item[$(1)$]\ $\displaystyle \chi_{\mathscr{A}(l)}^{quasi}(q)=
\begin{cases}
(q-2m+3+l)(q-m+1)\prod_{i=1}^{m-2}(q-2i+1)&\quad\text{if $q$ is odd},\\
f_l(q)\prod_{i=1}^{m-3}(q-2i)&\quad\text{if $q$ is even}.
\end{cases}$
\item[$(2)$]\ $\displaystyle \chi_{\mathscr{A}^{\prime}(l)}^{quasi}(q)=
\begin{cases}
(q-m+1+l)(q-m+2)\prod_{i=1}^{m-2}(q-2i+1)&\ \text{if $q$ is odd},\\
(q-m+1+l)\left(q^2-2(m-2)q+(m-1)(m-2)\right)\prod_{i=1}^{m-3}(q-2i)&\ \text{if $q$ is even}.
\end{cases}$
\end{itemize}
\end{Prop}
\begin{proof}
We prove the assertion by induction on $l$.
Let $l=0$. Then $\mathscr{A}(0)=\mathcal{D}_m$. It follows that
\begin{align*}
f_0(q)&=q^3-(4m-6)q^2+(m-1)(5m-8)q-(m-2)(2m^2-2m)\\
&=(q-2m+4)\left(q^2-2(m-1)q+m(m-1)\right).
\end{align*}
By Theorem \ref{thm-cox-chrquaipoly-ONB}, we have
\begin{align*}
\chi_{\mathscr{A}(0)}^{quasi}(q)&=
\begin{cases}
(q-m+1)\prod_{i=1}^{m-1}(q-2i+1)&\ \text{if $q$ is odd},\\
\left(q^2-2(m-1)q+m(m-1)\right)\prod_{i=1}^{m-2}(q-2i)&\ \text{if $q$ is even}
\end{cases}\\
&=
\begin{cases}
(q-2m+3)(q-m+1)\prod_{i=1}^{m-2}(q-2i+1)&\ \text{if $q$ is odd},\\
(q-2m+4)\left(q^2-2(m-1)q+m(m-1)\right)\prod_{i=1}^{m-3}(q-2i)&\ \text{if $q$ is even}
\end{cases}\\
&=
\begin{cases}
(q-2m+3)(q-m+1)\prod_{i=1}^{m-2}(q-2i+1)&\ \text{if $q$ is odd},\\
f_0(q)\prod_{i=1}^{m-3}(q-2i)&\ \text{if $q$ is even}.
\end{cases}
\end{align*}

Let $l>0$. We set $H=\{x_l-x_m=0\}$. Then $\mathscr{A}(l)_q=\mathscr{A}(l-1)_q\setminus\{H_q\}$.
By the induction hypothesis,
\begin{align*}
\chi_{\mathscr{A}(l-1)}^{quasi}(q)=
\begin{cases}
(q-2m+2+l)(q-m+1)\prod_{i=1}^{m-2}(q-2i+1)&\quad\text{if $q$ is odd},\\
f_{l-1}(q)\prod_{i=1}^{m-3}(q-2i)&\quad\text{if $q$ is even}
\end{cases}
\end{align*}
We have $\mathscr{A}(l-1)_q^{H_q}=\{\{2x_1=0\}_q\}\cup(\mathcal{D}_{m-1})_q$ and
\begin{align*}
\chi_{\{\{2x_1=0\}\}\cup\mathcal{D}_{m-1}}^{quasi}(q)=
\begin{cases}
(q-m+1)\prod_{i=1}^{m-2}(q-2i+1)&\ \text{if $q$ is odd},\\
(q^2-2(m-1)q+(m-2)(m+1))\prod_{i=1}^{m-3}(q-2i)&\ \text{if $q$ is even}
\end{cases}
\end{align*}
by Lemma \ref{lem-case-21...1and2}. Here
\begin{align*}
&f_{l-1}(q)+q^2-(2m-2)q+(m-2)(m+1)\\
=&\,q^3-(4m-5-l)q^2+(m-1)(5m-6-2l)q-(m-2)(2m^2-(l+1)m-k+1)\\
&+q^2-2(m-1)q+(m-2)(m+1)\\
=&\,q^3-(4m-6-l)q^2+(m-1)(5m-8-2l)q-(m-2)(2m^2-(l+2)m-l)
=f_l(q).
\end{align*}
Therefore
\begin{align*}
\chi_{\mathscr{A}(l)}^{quasi}(q)&=
\chi_{\mathscr{A}(l-1)}^{quasi}(q)+\chi_{\{\{2x_1=0\}\}\cup\mathcal{D}_{m-1}}^{quasi}(q)\\
&=
\begin{cases}
(q-2m+3+l)(q-m+1)\prod_{i=1}^{m-2}(q-2i+1)&\quad\text{if $q$ is odd},\\
f_l(q)\prod_{i=1}^{m-3}(q-2i)&\quad\text{if $q$ is even}.
\end{cases}
\end{align*}

$(2)$ Let $l=0$. Then $\mathscr{A}^{\prime}(0)=\mathscr{A}(m-1)$.
It follows that
\begin{align*}
f_{m-1}(q)&=q^3-(3m-5)q^2+3(m-1)(m-2)q-(m-2)(m-1)^2\\
&=(q-m+1)(q^2-2(m-2)q+(m-1)(m-2)).
\end{align*}
Therefore
\begin{align*}
\chi_{\mathscr{A}^{\prime}(0)}^{quasi}(q)
&=
\begin{cases}
(q-m+1)(q-m+2)\prod_{i=1}^{m-1}(q-2i+1)&\ \text{if $q$ is odd},\\
(q-m+1)(q^2-2(m-2)q+(m-1)(m-2))\prod_{i=1}^{m-3}(q-2i)&\ \text{if $q$ is even}.
\end{cases}
\end{align*}

Next let $l>0$ and we set $H=\{x_l+x_m=0\}$
Then $\mathscr{A}^{\prime}(l)_q=\mathscr{A}^{\prime}(l-1)_q\setminus\{H_q\}$ and $\displaystyle \mathscr{A}^{\prime}(l-1)_q^{H_q}=(\mathcal{D}_{m-1})_q$.
By Theorem \ref{thm-cox-chrquaipoly-ONB}, Lemma \ref{lem-inc-exc}, and the induction hypothesis, we have
\begin{align*}
\chi_{\mathscr{A}^{\prime}(l)}^{quasi}(q)&=
\chi_{\mathscr{A}^{\prime}(l-1)}^{quasi}(q)+\chi_{\mathcal{D}_{m-1}}^{quasi}(q)\\
=&\,
\begin{cases}
(q-m+l)(q-m+2)\prod_{i=1}^{m-2}(q-2i+1)&\ \text{if $q$ is odd},\\
(q-m+l)(q^2-2(m-2)q+(m-1)(m-2))\prod_{i=1}^{m-3}(q-2i)&\ \text{if $q$ is even}
\end{cases}\\
&+
\begin{cases}
(q-m+2)\prod_{i=1}^{m-2}(q-2i+1)&\ \text{if $q$ is odd},\\
((q^2-2(m-2)q+(m-1)(m-2))\prod_{i=1}^{m-3}(q-2i)&\ \text{if $q$ is even}
\end{cases}\\
=&\,
\begin{cases}
(q-m+1+l)(q-m+2)\prod_{i=1}^{m-2}(q-2i+1)&\ \text{if $q$ is odd},\\
(q-m+1+l)(q^2-2(m-2)q+(m-1)(m-2))\prod_{i=1}^{m-3}(q-2i)&\ \text{if $q$ is even}.
\end{cases}
\end{align*}
\end{proof}

\begin{Examp}
In the case when $m=4$, we list the characteristic quasi-polynomials for $\mathscr{A}(l)$ and $\mathscr{A}^{\prime}(l)$ defined in Proposition \ref{prop-chiD-deleting}:
\begin{align*}
\chi_{\mathscr{A}(0)}^{quasi}(q) &= 
\begin{cases}
\left(q - 5\right) \left(q - 3\right)^2 \left(q - 1\right) &\ (q\ \text{is odd}), \\
\left(q - 4\right)\left(q - 2\right)\left(q^2 -6q + 12\right) &\ (q\ \text{is even}). 
\end{cases}\\
\chi_{\mathscr{A}(1)}^{quasi}(q) &= 
\begin{cases}
\left(q - 4\right) \left(q - 3\right)^{2} \left(q - 1\right) &\ (q\ \text{is odd}), \\ 
\left(q - 2\right) \left(q^{3} - 9 q^{2} + 30 q - 38\right) &\ (q\ \text{is even}).
\end{cases}\\
\chi_{\mathscr{A}(2)}^{quasi}(q) &=
\begin{cases}
\left(q - 3\right)^{3} \left(q - 1\right) &\ (q\ \text{is odd}), \\
\left(q - 2\right) \left(q^{3} - 8 q^{2} + 24 q - 28\right) &\ (q\ \text{is even}). 
\end{cases}\\
\chi_{\mathscr{A}(3)}^{quasi}(q) &=
\begin{cases}
\left(q - 3\right)^{2} \left(q - 2\right) \left(q - 1\right) &\ (q\ \text{is odd}), \\
\left(q - 3\right) \left(q - 2\right) \left(q^{2} - 4 q + 6\right) &\ (q\ \text{is even}). 
\end{cases}\\
\chi_{\mathscr{A}^{\prime}(1)}^{quasi}(q) &=
\begin{cases}
\left(q - 3\right) \left(q - 2\right)^{2} \left(q - 1\right) &\ (q\ \text{is odd}), \\
\left(q - 2\right)^{2} \left(q^{2} - 4 q + 6\right) &\ (q\ \text{is even}). 
\end{cases}\\
\chi_{\mathscr{A}^{\prime}(2)}^{quasi}(q) &=
\begin{cases}
\left(q - 3\right) \left(q - 2\right) \left(q - 1\right)^{2} &\ (q\ \text{is odd}), \\
\left(q - 2\right) \left(q - 1\right) \left(q^{2} - 4 q + 6\right)&\ (q\ \text{is even}). 
\end{cases}\\
\chi_{\mathscr{A}^{\prime}(3)}^{quasi}(q) &=
\begin{cases}
q \left(q - 3\right) \left(q - 2\right) \left(q - 1\right) &\ (q\ \text{is odd}), \\
q \left(q - 2\right) \left(q^{2} - 4 q + 6\right) &\ (q\ \text{is even}). 
\end{cases}
\end{align*}
\end{Examp}

\section*{Acknowledgments}
This work was supported by JSPS KAKENHI Grant Number JP20K14282.


\begin{thebibliography}{10}
\bibitem{Athanasiadis1996} C. A. Athanasiadis. Characteristic polynomials of subspace arrangements and finite fields. \textit{Adv. Math.}, {\bf 122}, 193--233, (1996).
\bibitem{Beck-Zaslavsky} M. Beck and T. Zaslavsky. Inside-out polytopes. \textit{Adv. Math.}, {\bf 205} (1), 134--162, (2006).
\bibitem{Bibby-Delucchi2022} C. Bibby and E. Delucchi. Supersolvable posets and fiber-type abelian arrangements.
\textit{Selecta Math. (N.S.)}, {\bf 30}, 89, (2024).
\bibitem{Blass-Sagan} A. Blass and B. Sagan. Characteristic and Ehrhart polynomials. \textit{J. Algebraic Combin.}, {\bf 7}, 115--126, (1998).
\bibitem{HTY2112} A. Higashitani, T. N. Tran, and M. Yoshinaga. Period collapse in characteristic quasi-polynomials of hyperplane arrangements. \textit{Int. Math. Res. Not. IMRN}, {\bf 10}, 8934--8963 (2023).
\bibitem{KTT08} H. Kamiya, A. Takemura, and H. Terao. Periodicity of hyperplane arrangements with integral coefficients modulo positive integers. \textit{J. Algebr. Comb.}, {\bf 27} (3), 317--330, (2008).
\bibitem{KTT10-arxiv-v1} H. Kamiya, A. Takemura, and H. Terao. The characteristic quasi-polynomials of the arrangements of root systems. arXiv:0707.1381v1.
\bibitem{KTT10} H. Kamiya, A. Takemura, and H. Terao. The characteristic quasi-polynomials of the arrangements of root systems and mid-hyperplane arrangements. Arrangements, local systems and singularities, Progr. Math., {\bf 283},177--190 Birkh{\"a}user Verlag, Basel, (2010).
\bibitem{KTT11} H. Kamiya, A. Takemura, and H. Terao. Periodicity of non-Central integral arrangements modulo positive integers. \textit{Ann. Comb.}, {\bf 15} (3), 449--464, (2011).
\bibitem{Liu-Tran-Yoshinaga21} Y. Liu, T. N. Tran, and M. Yoshinaga. $G$-Tutte polynomials and abelian Lie group arrangements. \textit{Int. Math. Res. Not. IMRN}, {\bf 1}, 150--188 (2021).
\bibitem{PPTV2024} R. Pagaria, M. Pismataro, T. N. Tran, and L. Vecchi, Inductive and divisional posets. \textit{J. London Math. Soc.}, {\bf 109}, (2024).
\bibitem{Postnikov-Stanley} A. Postnikov and R. Stanley, Deformations of Coxeter hyperplane arrangements. \textit{J. Combin. Theory Ser. A}, {\bf 91}, 544--597, (2000).
\bibitem{Tamura} S. Tamura, Postnikov-Stanley Linial arrangement conjecture. \textit{J. Algebraic Combin.}, {\bf 58}, 651--679, (2023).
\bibitem{Tran-yoshinaga2019} T. N. Tran and M. Yoshinaga. Combinatorics of certain abelian Lie group arrangements and chromatic quasipolynomials.
\textit{J. Combin. Theory Ser. A}, {\bf 165}, 258--272 (2019).
\bibitem{Yoshinaga18b} M. Yoshinaga. Worpitzky partitions for root systems and characteristic quasi-polynomials. \textit{Tohoku Math. J.}, {\bf 70} (1), 39--63, (2018).
\bibitem{Yoshinaga18a} M. Yoshinaga. Characteristic polynomials of Linial arrangements for exceptional root systems. \textit{J. Combin. Theory Ser. A}, {\bf 157}, 267--286, (2018).
\end{thebibliography}
\end{document}